\documentclass[journal,twoside,web]{ieeecolor}
\usepackage{generic}
\usepackage{cite}
\usepackage{amsmath,amssymb,amsfonts}
\usepackage{algorithmic}
\usepackage{graphicx}
\usepackage{algorithm,algorithmic}

\usepackage{array}
\usepackage[caption=false,font=normalsize,labelfont=sf,textfont=sf]{subfig}
\usepackage{textcomp}
\usepackage{stfloats}
\usepackage{url}
\usepackage{verbatim}
\usepackage{bm}%字体加粗
\usepackage{multicol}%单双栏变化

\usepackage{color}
\usepackage{ulem}
\usepackage{multicol}%单双栏调整
\usepackage{multirow}
\usepackage{booktabs}%表格横线加粗。
\newtheorem{theorem}{Theorem}
\newtheorem{lemma}{Lemma}

\newtheorem{remark}{Remark}
\newtheorem{definition}{Definition}
\newtheorem{assumption}{Assumption}

\usepackage{threeparttable}%表格备注宏包
\usepackage{hyperref}
\hypersetup{hidelinks=true}

\begin{document}
\title{Cryptography-Based Privacy-Preserving Method for Distributed Optimization over Time-Varying Directed Graphs with Enhanced Efficiency}
%\title{A Sample Article Using IEEEtran.cls\\ for IEEE Journals and Transactions}

\author{Bing Liu,
        Furan Xie
        and Li Chai*% <-this % stops a space
\thanks{This work is supported by National Natural Science Foundation of China under grant number 62173259.}% <-this % stops a space
\thanks{Bing Liu and Furan Xie are with the Engineering Research Center of Metallurgical Automation and Measurement Technology, Wuhan University of Science and Technology, Wuhan 430081, China (e-mail: liubing17@wust.edu.cn; xiefuran328@wust.edu.cn).}% <-this % stops a space
\thanks{Li Chai is with the College of Control Science and Engineering, Zhejiang University, Hangzhou 310027, China (e-mail: chaili@zju.edu.cn).}
\thanks{*Corresponding Author.}}
\maketitle
\begin{abstract}
In this paper, we study the privacy-preserving distributed optimization problem, aiming to prevent attackers from stealing the private information of agents. For this purpose, we propose a novel privacy-preserving algorithm based on the Advanced Encryption Standard (AES), which is both secure and computationally efficient. By appropriately constructing the underlying weight matrices, our algorithm can be applied to time-varying directed networks. We show that the proposed algorithm can protect an agent’s privacy if the agent has at least one legitimate neighbor at the initial iteration. Under the assumption that the objective function is strongly convex and Lipschitz smooth, we rigorously prove that the proposed algorithm has a linear convergence rate. Finally, the effectiveness of the proposed algorithm is demonstrated by numerical simulations of the canonical sensor fusion problem.
\end{abstract}

\section{Introduction}
\IEEEPARstart{I}{n} recent years, distributed optimization has attracted extensive attention due to its wide application in many emerging and popular fields including multi-agent systems, machine learning, sensor networks, and power systems, etc \cite{bazerque2010distributed, tsianos2012consensusbased, tao2019asurvey, yi2020average, yi2023convergence}. In these applications, each agent has access to a local objective function and all agents cooperatively solve a global objective function that is the sum of their individual objective functions through local computation and communication. The distributed optimization problem can be generally formulated as follows:
\begin{equation}\label{prob_orig_0}
\mathop{\min}_{x\in\mathbb{R}^d} f(x) = \frac{1}{m}\sum_{i=1}^{m}f_i(x),
\end{equation}
where $f_i: \mathbb{R}^d \to \mathbb{R}$ is the local objective function of agent $i$, $x \in \mathbb{R}^d$ is the global decision variable, and $m$ is the number of agents.

%in a peer-to-peer architecture
To solve the problem (\ref{prob_orig_0}), plenty of first-order distributed optimization algorithms have been reported over different underlying communication networks/graphs. References \cite{duchi2012dual, jakovetic2014fast, shi2014on, shi2015extra, olshevsky2017linear, qu2018harnessing, liu2022convergence} focus on static undirected graphs and \cite{nedic2009distributed, xu2018convergence, nedic2017achieving} address the time-varying/stochastic undirected graphs. Since in many scenarios the agent's communications are directed, a series of algorithms have been developed over directed graphs \cite{zeng2017extrapush, xi2017dextra, xi2018linear, xin2018a, xi2018addopt, xin2018frost, pu2021pushpull}, and even over time-varying directed graphs \cite{nedic2015distributed, nedic2017achieving, saadatniaki2020decentralized, nedic2022abpushpull}. On the other hand, for the algorithms over undirected graphs, their underlying weight matrices are usually double stochastic, but it is unrealistic to construct double stochastic weight matrices over directed graphs. To overcome this difficulty, \cite{nedic2015distributed, nedic2017achieving, zeng2017extrapush, xi2017dextra, xi2018addopt} develop algorithms based on the push-sum protocol using only column stochastic weight matrices. These algorithms require each agent to know its out-degree to construct column stochastic weight matrices. In contrast, the algorithms in \cite{xi2018linear, xin2018frost} only use row stochastic weight matrices and require no such knowledge. Using only column or row stochastic weight matrices will cause an imbalance. These algorithms need to additionally estimate the right or left eigenvectors with eigenvalue 1 of weight matrices to alleviate the imbalance. The algorithms in \cite{xin2018a, saadatniaki2020decentralized, pu2021pushpull, nedic2022abpushpull} avoid this extra estimation by using a row stochastic weight matrix as well as a column stochastic weight matrix.

%However, does state and/or gradient sharing leak sensitive data of the agents?
The research interests of the above literature focus on the convergence of distributed algorithms over different communication graphs, ignoring the privacy issues of the algorithms. In the agent's communications of their algorithms, each agent explicitly shares its state and/or gradient estimation with neighboring agents at each iteration. A natural question is whether the state and/or gradient sharing will leak sensitive data of the agents. The results in \cite{zhu2019deep} show that in distributed training and collaborative learning, the agents can precisely restore the raw training data of neighboring agents based on their shared gradients (e.g., images in computer vision tasks or texts in natural language processing tasks). This is undesirable in many applications involving sensitive data. References \cite{burbano2019inferring, alanwar2017proloc} also reveal the privacy issues in distributed optimization. In fact, without an effective privacy preservation mechanism, the attacker can easily steal the privacy of agents in distributed optimization.

%the higher the desired privacy level, the larger the noises required
Recently, designing appropriate privacy preservation mechanisms for distributed optimization has been of considerable interest and many privacy-preserving algorithms have been proposed, broadly divided into three categories: differential privacy-based methods \cite{huang2015differentially, cortes2016differential, han2017differentially, nozari2018differentially, xiong2020privacy, cupta2021preserving, ding2022differentially, zhou2023private, wang2023tailoring}, structure-based methods \cite{yan2013distributed, lou2018privacy, gade2018private, gao2023dynamics}, and cryptography-based methods \cite{freris2016distributed, yang2018privacy, ruan2019secure, zhang2019admm, zhang2019enabling, wang2019privacypreservtion, hadjicostis2020privacypreserving, alexandru2021cloudbased, yan2021distributed, hou2023encrypted, chen2023privacypreserving}. Differential privacy-based methods obscure the raw data by injecting random noise to the exchanged information between agents \cite{huang2015differentially, cortes2016differential, han2017differentially, xiong2020privacy, ding2022differentially, wang2023tailoring} or to the local objective function of agents \cite{nozari2018differentially, cupta2021preserving, zhou2023private}, thereby protecting the privacy of agents. However, for such methods, the higher the desired level of privacy, the larger the variance of required noise, which unavoidably compromises the optimization accuracy or reduces the convergence rate of the original algorithm. Structure-based methods exploit the structural properties of distributed optimization and inject correlated randomization into the algorithm to enable privacy preservation without compromising the optimization accuracy (such as the constant uncertain parameter in \cite{yan2013distributed, lou2018privacy} and ``structured'' noise in \cite{gade2018private, gao2023dynamics}). However, since the injected randomization is correlated, such methods cannot guarantee strong privacy preservation.

The third category of methods is based on cryptography (such as the Paillier cryptosystem \cite{paillier1999publickey} and the Advanced Encryption Standard (AES) \cite{rijmen2001advanced}). Cryptography-based methods enable privacy preservation in distributed optimization by encrypting the exchanged information between agents. The advantage of cryptography-based methods is that they can provide strong privacy preservation without compromising the optimization accuracy and the convergence rate. References \cite{freris2016distributed, yang2018privacy, hadjicostis2020privacypreserving, alexandru2021cloudbased, hou2023encrypted} have directly used the Paillier cryptosystem to develop privacy-preserving algorithms for different distributed optimization problems, but all require an aggregator (such as the cloud in \cite{freris2016distributed, alexandru2021cloudbased} and the system operator in \cite{yang2018privacy, hou2023encrypted}) or assume the presence of a trusted agent \cite{hadjicostis2020privacypreserving}. The reason for the limitation is that the cryptosystem cannot be directly incorporated into fully distributed algorithms to guarantee both optimization and privacy. Reference \cite{ruan2019secure} first proposes a cryptography-based privacy-preserving protocol to overcome this limitation, which can solve the average consensus problem in a fully distributed manner. In this protocol, each agent generates its own pair of keys and asks its neighbors to use its public key to encrypt the required information, and the weight is split into two random positive numbers maintained by corresponding agents to protect the privacy of agents. Inspired by the protocol, references \cite{zhang2019admm, zhang2019enabling, wang2019privacypreservtion} present privacy-preserving distributed algorithms integrating the Paillier cryptosystem to solve problem (\ref{prob_orig_0}) over undirected graphs based on the alternating direction method of multipliers (ADMM), projected subgradient method, and online dual averaging method, respectively. References \cite{yan2021distributed, chen2023privacypreserving} also adopt a similar strategy to develop cryptography-based privacy-preserving distributed algorithms for the economic dispatch problem.

Although cryptography-based privacy-preserving algorithms in \cite{zhang2019admm, zhang2019enabling, wang2019privacypreservtion, yan2021distributed, chen2023privacypreserving} have been successful in solving problem (\ref{prob_orig_0}) in a fully distributed manner, there are still serious issues that have not been solved. First, these algorithms can only be applied to undirected graphs since neighboring agents need to perform bidirectional communication. This significantly limits the deployment of these algorithms. Secondly, these algorithms adopt the Paillier cryptosystem, which will impose serious computational burdens in the implementation. In addition, neighboring agents need to exchange information twice per iteration, which also increases the communication overhead. Thirdly, most privacy-preserving algorithms require a relatively strict condition that each agent has at least one legitimate neighbor at all time instance. In response to the above issues, we propose a series of solutions, which are summarized as follows together with the contributions of this paper.
%Third, there is no unified criterion for evaluating the computational and communication overheads brought by the cryptosystem to the algorithm, making it inconvenient to compare the complexity of different cryptography-based privacy-preserving algorithms.

1) \textit{Extension to time-varying directed graphs and linear convergence rate:} Existing cryptography-based privacy-preserving algorithms in distributed optimization can only be applied to undirected graphs. In this paper, we propose a novel cryptography-based distributed algorithm for time-varying directed graphs by appropriately constructing the underlying weight matrices. Instead of the sub-linear convergence rate of algorithms in \cite{zhang2019admm, zhang2019enabling}, we prove that the proposed algorithm has a linear convergence rate under the assumption that the objective function is strongly convex and Lipschitz smooth.
%This is due to the limitation of the underlying weight matrices.
%Under the assumption that the objective function is strongly convex and Lipschitz smooth, we rigorously prove that the proposed algorithm has a linear convergence rate when the step-size is sufficiently small.

2) \textit{Reduction of the computational and communication overheads:} Existing cryptography-based privacy-preserving algorithms for distributed optimization adopt the Paillier cryptosystem and exploit its property of partially homomorphic. Our proposed algorithm utilizes AES, which significantly reduces the computational overhead. Furthermore, AES also has advantages in security, accuracy, and hardware \& software implementation. Moreover, in the proposed algorithm, neighboring agents only exchange information once per iteration instead of twice in the existing algorithms, which reduces the communication overhead.
%In addition, since the proposed algorithm has a linear convergence rate instead of a sub-linear convergence rate \cite{zhang2019admm, zhang2019enabling}, our algorithm requires fewer iterations to reach convergence, which further reduces the computational and communication overheads.

%3) \textit{Theoretical analysis of the computational and communication complexities:} We define two concepts of computational complexity and communication complexity for cryptography-based privacy-preserving algorithms to represent the computational and communication overheads brought by cryptosystems to the algorithms. Then, we analyze the computational and communication complexities of the proposed algorithm and the existing state-of-the-art algorithms. Our results show that compared with other state-of-the-art algorithms, the encryption and decryption complexities of our algorithm are reduced by an order of magnitude without increasing the complexity of other operations as well as the communication complexity. This also theoretically confirms that our algorithm can reduce the computational overhead.
%The computational complexity of our algorithm is $\mathcal{O}(Kd)$ encryption and $\mathcal{O}(Kd)$ decryption (the computational complexity of the algorithms in \cite{} is $\mathcal{O}(Kdm)$ encryption and $\mathcal{O}(Kdm)$ decryption), where $K$ is the number of iterations, $d$ is the dimension of the decision variable, and $m$ is the number of the agents.

3) \textit{Privacy preservation under more relaxed conditions:}
The proposed distributed algorithm can preserve the privacy of agents against both the honest-but-curious adversary and the external eavesdropper by designing a weight generation rule and utilizing AES. It only requires that the agents have at least one legitimate neighbor at the initial iteration instead of at all time instance in \cite{zhang2019admm, gao2023dynamics}.

The rest of the paper is organized as follows. In Section \ref{section_problem}, we introduce the private distributed optimization problem and AES. In Section \ref{section_algorithm}, we propose a cryptography-based privacy-preserving distributed algorithm. Then, we establish the linear convergence rate of the proposed algorithm in Section \ref{section_convergence}. In Section \ref{section_privacy_analysis}, we analyze the privacy of the proposed algorithm. Finally, we conduct numerical simulations in Section \ref{section_numerical_simulation} and conclude the paper in Section \ref{section_conclusion}.

%Given an integer $n \ge 2$, define the set of integers $\{0, 1, \dots, n-1\}$ to be $\mathbb{Z}_n$. 
\textit{Notations}: For a time-varying directed graph sequence $\{ \mathcal{G}(0), \mathcal{G}(1), \cdots \}$, each graph $\mathcal{G}(k)$ at iteration $k$ is specified by a static set of agents $\mathcal{V} = \{1, 2, \dots, m\}$ and a time-varying set of edges $\mathcal{E}(k)$ in which $(i, j) \in \mathcal{E}(k)$ if and only if agent $i$ can receive information from agent $j$ at iteration $k$. For each agent $i$, its in-degree and out-degree neighbor sets at iteration $k$ are denoted as $\mathcal{N}_i^{\text{in}}(k)$ and $\mathcal{N}_i^{\text{out}}(k)$, respectively. Correspondingly, the in-degree and out-degree of agent $i$ at iteration $k$ are denoted as $d_i^{\text{in}}(k) = | \mathcal{N}_i^{\text{in}}(k) |$ and $d_i^{\text{out}}(k) = | \mathcal{N}_i^{\text{out}}(k) |$, respectively. We use $\mathbf{1}$ and $I$ to denote the column vector with all entries of 1 and the identity matrix respectively, whose size is determined by the context. For a given vector $x\in\mathbb{R}^{d}$, we denote the Euclidean norm of $x$ as $\| x \|_2$. For any matrix $\mathbf{x}\in\mathbb{R}^{m\times d}$, we denote the Frobenius norm of $\mathbf{x}$ as $\| \mathbf{x} \|_\text{F}$, the max norm of $\mathbf{x}$ as $\| \mathbf{x} \|_{\text{max}} = \max_{ij}|x_{ij}|$, and the spectral norm of $\mathbf{x}$ as $\| \mathbf{x} \|_{2}$. For any $m\times d$ matrices $\mathbf{x}$ and $\mathbf{y}$, we denote the inner product of $\mathbf{x}$ and $\mathbf{y}$ as $\left<\mathbf{x}, \mathbf{y} \right> = \text{Trace}(\mathbf{x}^T\mathbf{y})$.
%For a matrix $\mathbf{x}\in\mathbb{R}^{m\times d}$, its average across rows is defined as $\bar{x} = \frac{1}{m}\mathbf{x}^T\mathbf{1} \in \mathbb{R}^d$ and its consensus violation is defined as $\check{\mathbf{x}} = \mathbf{x} - \mathbf{1}\bar{x}^T = (I - \frac{1}{m}\mathbf{1}\mathbf{1}^T)\mathbf{x} = \mathbf{R}\mathbf{x}$ where $\mathbf{R} = I - \frac{1}{m}\mathbf{1}\mathbf{1}^T$. We also define the $\mathbf{R}$ weighted norm of matrix $\mathbf{x}$ as $\| \mathbf{x} \|_\text{R} = \sqrt{\left<\mathbf{x}, \mathbf{R}\mathbf{x} \right>}$. Due to $\mathbf{R} = \mathbf{R}^T \mathbf{R}$, we have $\| \mathbf{x} \|_{\text{R}} = \| \mathbf{R}\mathbf{x} \|_{\text{F}} = \|  \check{\mathbf{x}} \|_{\text{F}}$.

\section{Problem Formulation and Preliminaries}\label{section_problem}
\subsection{Privacy-Preserving Distributed Optimization Problem}
We consider a system with $m$ agents in which each agent $i$ is associated with a local objective function $f_i: \mathbb{R}^d \to \mathbb{R}$ and can communicate with its neighbors through a time-varying directed communication network. All agents of the system aim to cooperatively solve the problem (\ref{prob_orig_0}) by maintaining a local copy $x_i\in\mathbb{R}^d$ of the global decision variable $x$ while protecting their private information from being disclosed. Such a privacy-preserving distributed optimization problem is formulated as follows:
\begin{subequations}\label{prob_aggr}
\begin{gather}
\mathop{\min}_{x_i\in\mathbb{R}^d, i\in\mathcal{V}} f(x) = \frac{1}{m}\sum_{i=1}^{m}f_i(x_i),\\
\text{s.t.}\ x_1=x_2=\cdots=x_m.
\end{gather}
\end{subequations}
%where $\mathbf{x}=[x_1, x_2, \cdots, x_m]^T \in\mathbb{R}^{m\times d}$ is a matrix and each row $i$ of $\mathbf{x}$ is agent $i$'s local decision variable.

%To solve the aggregate optimization problem (\ref{prob_aggr}), we make the following assumption about the objective functions $f_i$.
%\begin{assumption}\label{assu_func}
%\textit{The local objective functions $f_i$ of all agents are convex and continuously differentiable.}
%\end{assumption}

We consider two widely-used attack models: the \textit{Honest-but-curious adversary} and the \textit{External eavesdropper}. In the following, we give the definitions of these two attackers as well as the legitimate agent.
\begin{definition}\label{definition_honest_attack}
\cite{goldreich2009foundations} (Honest-but-curious adversary) The honest-but-curious adversaries are participating agents who execute correct updates but collect information from neighboring agents, trying to infer their private information independently or in collusion.
\end{definition}

\begin{definition}\label{definition_external_attack}
\cite{goldreich2009foundations} (External eavesdropper) The external eavesdroppers are adversaries who know the updates of each agent but are not participating agents and can wiretap the exchanged information on all communication links, trying to infer the agents' private information.
\end{definition}

\begin{definition}
(Legitimate agent) The legitimate agents are participating agents who execute correct updates and have no intention to infer other agents' private information.
\end{definition}

\subsection{Advanced Encryption Standard}\label{subsection_cryptosystem}
AES is a symmetric block cipher algorithm that uses the same key to encrypt and decrypt data in blocks of 128 bits \cite{rijmen2001advanced}. The key length for AES can be 128, 192, and 256 bits. AES encryption converts data into an unintelligible form called ciphertext. The encryption process involves four transformation functions: Substitution Bytes, Shift Rows, Mix Columns, and Add Round Key. AES decryption converts the ciphertext back to its original form, called plaintext. The decryption process includes the inverse operations of those four transformation functions.

For convenience, we define the symbol $[\![ \cdot ]\!]$ to be the encryption operation of AES, that is, for plaintext $z$, the resulting ciphertext after AES encryption is $[\![ z ]\!]$.

\section{Privacy-Preserving Algorithm Design}\label{section_algorithm}
In this section, we aim to develop a cryptography-based privacy-preserving algorithm for the distributed optimization problem (\ref{prob_aggr}) over time-varying directed graphs. We first give the original updates of the algorithm and then integrate AES into the original updates to obtain the final privacy-preserving distributed algorithm. In the original updates, each agent $i$ maintains four variables $y_i(k)$, $w_i(k)$, $x_i(k)$, and $s_i(k)$ with initialization $y_i(0) = x_i(0) \in\mathbb{R}^{d}$, $w_i(0) \in \mathbb{R}$, $s_i(0) = \nabla f_i\big( x_i(0) \big) \in \mathbb{R}^d$, and $w_i(k)$ is reset to 1 when $k=1$. At each iteration $k$, each agent $i$ ($i \in \mathcal{V}$) sends weighted information $a_{li}(k)y_i(k)$, $a_{li}(k)s_i(k)$, and $a_{li}(k)w_i(k)$ to its out-neighbors $l\in\mathcal{N}_i^{\text{out}}(k)$, and receives $a_{ij}(k)y_j(k)$, $a_{ij}(k)s_j(k)$, and $a_{ij}(k)w_j(k)$ from its in-neighbors $j \in \mathcal{N}_i^{\text{in}}(k)$. After the information exchange step, each agent $i$ executes the original updates as follows: 
\begin{subequations}\label{init_upda}
	\begin{align}
	&y_i(k+1) = \sum_{j=1}^{m} a_{ij}(k) \big( y_j(k) - \eta s_j(k) \big), \label{update_a}\\
	&w_i(k+1) = \sum_{j=1}^{m} a_{ij}(k)w_j(k), \label{update_b}\\
	&x_i(k+1) = y_i(k+1)/w_i(k+1), \label{update_c}\\
	&s_i(k+1) = \sum_{j=1}^{m} a_{ij}(k)s_j(k) + \nabla f_i\big( x_i(k+1) \big) - \nabla f_i\big( x_i(k) \big), \label{update_d}
	\end{align}
\end{subequations}
where $\nabla f_i(\cdot)$ is the gradient of $f_i$, $\eta > 0$ is a fixed step-size, and the time-varying weight matrix $A(k) = \big[a_{ij}(k)\big]\in\mathbb{R}^{m\times m}$ satisfies $\mathbf{1}^TA(k) = \mathbf{1}^T$ for all $k\ge 0$ and every nonzero entry in $A(k)$ is not less than a positive constant $c_0$ for any $k \ge 1$ where $0 < c_0 < 1/m$. We give the rule for each agent $i$ to generate weights $a_{li}(k), l\in\mathcal{V}, k \ge 0$ in Table \ref{table_weight_rule}.

\begin{table}[!h]\centering
\caption{The weight generation rule for each agent $i$.}
\begin{tabular}{| c | c |}
\hline
$k = 0$ &
$ a_{li}(k) \left\{ \begin{array}{ll}
\in \mathbb{R}, & l\in\mathcal{N}_i^{\text{out}}(k), \\
= 0, & l\notin \mathcal{N}_i^{\text{out}}(k)\cup \{i\},\\
= 1-\sum_{l\in\mathcal{V}, l\not=i}a_{li}(k), & l=i.
\end{array} \right. $ \\
\hline
$k \ge 1$ & $ a_{li}(k) \left\{ \begin{array}{ll} \in \left[ c_0, \frac{1 - c_0}{d_i^{\text{out}}(k)} \right], & l\in\mathcal{N}_i^{\text{out}}(k), \\
= 0, & l\notin \mathcal{N}_i^{\text{out}}(k)\cup \{i\},\\
= 1-\sum_{l\in\mathcal{V}, l\not=i}a_{li}(k), & l=i.
\end{array}
\right.
$ \\
\hline
\end{tabular}
\label{table_weight_rule}
\end{table}

Notice that, the weight $a_{li}(k)$ generated by agent $i$ is randomly sampled with a uniform distribution at each iteration: when $k = 0$, $a_{li}(k), l\in\mathcal{N}_i^{\text{out}}(k)$ is randomly selected from $\mathbb{R}$, which can be positive, negative or zero. For $k \ge 1$, $a_{li}(k), l\in\mathcal{N}_i^{\text{out}}(k)$ is selected from a given range, which depends only on a fixed positive constant $c_0$ and the out-degree $d_i^{\text{out}}(k)$ of agent $i$. Such a weight generation rule can make $A(k)$ satisfy the corresponding conditions. It is also the key to protecting the privacy of agents from the honest-but-curious adversary (we will give a detailed analysis in Section \ref{section_privacy_analysis}). 
%In fact, without a proper weight generation rule, the original updates (\ref{init_upda}) and the distributed optimization algorithms such as Push-DIGing \cite{nedic2017achieving} and ADD-OPT \cite{xi2018addopt} that use column stochastic weight matrices cannot preserve the privacy of agents (see Remark \ref{remark_pushdiging} in Section \ref{section_privacy_analysis}).
 
\begin{algorithm}[!h]
\caption{AES-based Privacy-preserving Distributed Algorithm}\label{alg:alg_private}
\begin{algorithmic}[1]
\REQUIRE
Generate an AES key and distribute it to all agents. Each agent $i$ chooses arbitrary $x_i(0)=y_i(0)\in\mathbb{R}^d$ and $w_i(0) \in \mathbb{R}$, and sets $s_i(0) = \nabla f_i\big(x_i(0)\big)$.

\ENSURE

\STATE\label{step_1} Agent $i$ generates random weights $a_{li}(k)$ based on the weight generation rule in Table \ref{table_weight_rule} and uses the key to encrypt $[\![ a_{li}(k)y_i(k) ]\!]$, $[\![ a_{li}(k)s_i(k) ]\!]$, and $[\![ a_{li}(k)w_i(k) ]\!]$ for $l\in\mathcal{N}_i^{\text{out}}(k)$.
\STATE\label{step_2} Agent $i$ sends ciphertexts $[\![ a_{li}(k)y_i(k) ]\!]$, $[\![ a_{li}(k)s_i(k) ]\!]$, and $[\![ a_{li}(k)w_i(k) ]\!]$ to out-neighbors $l\in\mathcal{N}_i^{\text{out}}(k)$ and receives ciphertexts $[\![ a_{ij}(k)y_j(k) ]\!]$, $[\![ a_{ij}(k)s_j(k) ]\!]$, and $[\![ a_{ij}(k)w_j(k) ]\!]$ from in-neighbors $j\in\mathcal{N}_i^{\text{in}}(k)$.
\STATE\label{step_3} Agent $i$ uses the key to decrypt ciphertexts $[\![ a_{ij}(k)y_j(k) ]\!]$, $[\![ a_{ij}(k)s_j(k) ]\!]$, and $[\![ a_{ij}(k)w_j(k) ]\!]$ obtaining $a_{ij}(k)y_j(k)$, $a_{ij}(k)s_j(k)$, and $a_{ij}(k)w_j(k)$ for $j\in\mathcal{N}_i^{\text{in}}(k)$, and executes the updates (\ref{init_upda}) to obtain $y_i(k+1)$, $w_i(k+1)$, $x_{i}(k+1)$, and $s_{i}(k+1)$.
\end{algorithmic}
\label{alg_private}
\end{algorithm}

On the other hand, since the external eavesdropper can steal all the information on the communication links, the updates (\ref{init_upda}) combined with the weight generation rule in Table \ref{table_weight_rule} cannot prevent the external eavesdropper (see Remark \ref{remark_no_crytptosystem} in Section \ref{section_privacy_analysis}). Thus in Algorithm \ref{alg_private}, we integrate the updates (\ref{init_upda}) with AES to prevent the external eavesdropper. It is noted that existing cryptography-based privacy-preserving methods in distributed optimization utilize the asymmetric Paillier cryptosystem with homomorphic addition, while Algorithm \ref{alg_private} employs the symmetric AES. The reason Algorithm \ref{alg_private} can use a symmetric encryption is that the agents send weighted information rather than explicit information to neighboring agents. In comparison to the Paillier cryptosystem, AES has significant advantages in terms of security, accuracy, hardware \& software implementation, and especially computational cost \cite{mahajan2013study}. In Section \ref{subsection_simulation_comparison}, we will give detailed numerical comparisons.

\begin{remark}
In initialization, Algorithm 1 needs to generate and distribute a shared key to all agents, which is called key distribution in modern cryptography. The key distribution can be implemented based on either symmetric or asymmetric encryption \cite{stallings2013cryptography}, the Diffie-Hellman algorithm \cite{diffie1976new}, and quantum encryption \cite{charles2014quantum}, etc. Some methods require a trusted third party and some can be implemented in a distributed manner. We assume that the shared key is well distributed using existing methods to all agents in the initialization.
\end{remark}
%There are already various key distribution methods based on different strategies, including the key distribution method based on symmetric encryption \cite{stallings2013cryptography}, the method based on asymmetric encryption \cite{stallings2013cryptography}, the Diffie-Hellman algorithm \cite{diffie1976new}, and quantum key distribution methods \cite{charles2014quantum}, etc.

\section{Convergence Analysis}\label{section_convergence}
In this section, we theoretically prove the linear convergence rate of Algorithm \ref{alg_private}. We first rewrite the updates (\ref{init_upda}) into an augmented form: 
	\begin{subequations}
	\begin{align}
	\mathbf{w}(k+1)  =\ &A(k)\mathbf{w}(k), \label{update_compact_w}\\
	\mathbf{x}(k+1) =\ &\Phi(k)\big( \mathbf{x}(k) - \eta \mathbf{u}(k) \big), \label{update_2b}\\
	\mathbf{u}(k+1) =\ &\Phi(k)\mathbf{u}(k) \nonumber \\
	+&W(k+1)^{-1}(\nabla \mathbf{f}(\mathbf{x}(k+1)) - \nabla \mathbf{f}(\mathbf{x}(k))) \label{update_2c},
	\end{align}
	\end{subequations}
where,
	\begin{align*}
&\Phi(k) = W(k+1)^{-1}A(k)W(k), \\
&W(k+1) = \text{diag}\{\mathbf{w}(k+1)\}, \\
&\mathbf{u}(k) = W(k)^{-1}\mathbf{s}(k), \\
&\mathbf{w}(k) = [w_1(k), \cdots, w_m(k)]^T \in \mathbb{R}^m, \\
&\mathbf{x}(k) = [x_1(k), \cdots, x_m(k)]^T \in \mathbb{R}^{m \times d}, \\
&\mathbf{s}(k) = [s_1(k), \cdots, s_m(k)]^T\in\mathbb{R}^{m \times d}, \\
&\nabla \mathbf{f}(\mathbf{x}(k)) = [\nabla f_1(x_1(k)), \cdots, \nabla f_m(x_m(k))]^T \in \mathbb{R}^{m \times d}.
	\end{align*}
	
%and $\Phi_{b}(k) = I$ for any needed $k<0$
We also denote the notation $\Phi_{b}(k)$ as $\Phi_{b}(k) = \Phi(k)\Phi(k-1)\cdots \Phi(k-b+1), k = 0, 1, \cdots, b = 1, 2, \cdots, k+1$. For any matrix $\mathbf{a}\in \mathbb{R}^{m\times d}$, we denote the average across columns of $\mathbf{a}$ as $\bar{a} = \frac{1}{m}\mathbf{a}^T\mathbf{1} \in \mathbb{R}^d$ and the consensus violation of $\mathbf{a}$ as $\check{\mathbf{a}} = \mathbf{a} - \mathbf{1}\bar{a}^T = (I - \frac{1}{m}\mathbf{1}\mathbf{1}^T)\mathbf{a} = \mathbf{R}\mathbf{a}$ where $\mathbf{R} = I - \frac{1}{m}\mathbf{1}\mathbf{1}^T$. We denote the $\mathbf{R}$ weighted norm of matrix $\mathbf{a}$ as $\| \mathbf{a} \|_\text{R} = \sqrt{\left<\mathbf{a}, \mathbf{R}\mathbf{a} \right>}$. Due to $\mathbf{R} = \mathbf{R}^T \mathbf{R}$, we have $\| \mathbf{a} \|_{\text{R}} = \| \mathbf{R}\mathbf{a} \|_{\text{F}} = \|  \check{\mathbf{a}} \|_{\text{F}}$. For a sequence $\{ \mathbf{a}(0), \mathbf{a}(1), \cdots \}$ with $\mathbf{a}(k)\in \mathbb{R}^{m\times d}$, we denote $
	\| \mathbf{a} \|_{\text{F}}^{\theta, K} = \max_{k = 1, \cdots, K} \frac{1}{\theta^k}\| \mathbf{a}(k) \|_{\text{F}}
	$
where parameter $\theta \in (0, 1)$. If $\mathbf{a}(k)\in \mathbb{R}^{d}$, $\| \mathbf{a} \|_{\text{F}}^{\theta, K} = \max_{k = 1, \cdots, K} \frac{1}{\theta^k}\| \mathbf{a}(k) \|_{\text{2}}$. For the convenience of analysis, we denote $\mathbf{r}(k) = \mathbf{x}(k) - \mathbf{x}^*$ for $k \ge 0$ where $\mathbf{x}^* = \mathbf{1}{x^*}^T$ and $x^*$ is the optimal solution of problem (\ref{prob_orig_0}), $\mathbf{v}(k) = \nabla \mathbf{f}(\mathbf{x}(k)) - \nabla \mathbf{f}(\mathbf{x}(k-1))$ for $k \ge 1$ and $\mathbf{v}(0) = \mathbf{0}$ when $k=0$, $\check{\mathbf{u}}(k) = \mathbf{u}(k) - \mathbf{1}\bar{u}(k)^T$, and $\check{\mathbf{x}}(k) = \mathbf{x}(k) - \mathbf{1}\bar{x}(k)^T$.

Our main strategy is first to establish the following inequalities with respect to $\|\mathbf{r}\|_{\text{F}}^{\theta,K}$, $\| \mathbf{v} \|_{\text{F}}^{\theta, K}$, $\| \check{\mathbf{u}} \|_{\text{F}}^{\theta, K}$, and $\| \check{\mathbf{x}} \|_{\text{F}}^{\theta, K}$:
	\begin{equation}\begin{aligned}
	& \| \mathbf{v} \|_{\text{F}}^{\theta, K} \le \gamma_1 \|\mathbf{r}\|_{\text{F}}^{\theta,K}+b_1, \quad
	\| \check{\mathbf{u}} \|_{\text{F}}^{\theta, K} \le \gamma_2 \| \mathbf{v} \|_{\text{F}}^{\theta, K} + b_2, \\
	& \| \check{\mathbf{x}} \|_{\text{F}}^{\theta, K} \le \gamma_3 \| \check{\mathbf{u}} \|_{\text{F}}^{\theta, K} + b_3, \quad
	\|\mathbf{r}\|_{\text{F}}^{\theta,K} \le \gamma_4 \| \check{\mathbf{x}} \|_{\text{F}}^{\theta, K} + b_4,
	\end{aligned}\end{equation}
where $\gamma_1, \gamma_2, \gamma_3, \gamma_4$ are non-negative gain constants and $b_1, b_2, b_3, b_4$ are constants. Then, by using the small gain theorem \cite{nedic2017achieving}, we can derive that $\| \mathbf{r}(k) \|_{\text{F}}$, i.e., $\|\mathbf{x}(k) - \mathbf{x}^*\|_{\text{F}}$ converges at a linear rate if $\gamma_1\gamma_2\gamma_3\gamma_4 < 1$ holds.
\subsection{Basic Assumptions and Supporting Lemmas}
We first present some basic assumptions, on which Algorithm \ref{alg_private} has a linear convergence rate over time-varying directed graphs.
\begin{assumption}\label{smooth} 
The local objective function of each agent $i$ is $L_i$-smooth, that is, $$\| \nabla f_i(x) - \nabla f_i(y) \|_{2} \le L_i \| x-y \|_{2},\ \forall x,y\in\mathbb{R}^d,$$ where $L_i>0$ is the Lipschitz constant of $\nabla f_i (\cdot)$.
\end{assumption}

\begin{assumption}\label{convex}
The local objective function of each agent $i$ is $\mu_i$-strongly convex, that is, $$f_i(y) \ge f_i(x) + \nabla f_i(x)^T(y-x) + \frac{\mu_i}{2}\| y-x \|_{2}^2,\ \forall x,y\in\mathbb{R}^d,$$ where $\mu_i \in [0, +\infty)$ and $\sum_{i=1}^{m} \mu_i > 0$.
\end{assumption}

\begin{assumption}\label{connection}
The graph sequence $\{ \mathcal{G}(k) \}$ is uniformly strongly connected, that is, there exists a positive integer $\tilde{B}$ such that for any $t\ge 0$,
	\begin{equation}
	\mathcal{G}_{\tilde{B}}(t \tilde{B}) =\left\{\mathcal{V}, \bigcup_{l = t \tilde{B}}^{t \tilde{B} + \tilde{B} - 1} \mathcal{E}(l) \right\},
	\end{equation}
is strongly connected.
\end{assumption}

We denote $\hat{L} = \max_i\{L_i\}$ and $\bar{L} = (1/m)\sum_{i=1}^{m}L_i$, which are the Lipschitz constants of $\nabla \bold{f(x)}$ and $\nabla f(x)$ under Assumption \ref{smooth}, respectively. Similarly, we denote $\hat{\mu} = \max_i\{\mu_i\}$ and $\bar{\mu} = (1/m)\sum_{i=1}^{m}\mu_i$. Then, $f(x)$ is $\bar{\mu}$-strongly convex under Assumption \ref{convex}, which implies that problem (\ref{prob_orig_0}) has a unique optimal solution $x^*$. We also denote $\kappa = \hat{L}/\bar{\mu}$. Let $B =  2 \tilde{B} - 1$, then the $\mathcal{G}_B(k)$ is strongly connected for any $k \ge 0$ based on Assumption \ref{connection}.

To present the main results clearly, we need more notations. Recall that $c_0$ is the parameter to constrain the weights in Table \ref{table_weight_rule}, and $m$ is the number of agents. Denote $\sigma = (c_0)^{2+m B}$ and $\epsilon = 2m (1+\sigma^{-m B})/(1 - \sigma^{m B})$. Let $B_0 \ge B$ be such that
	\begin{align}\label{equation_varepsilon}
	\varepsilon = \epsilon (1 - \sigma^{m B})^{(B_0 - 1)/(m B)} < 1.
	\end{align}
There exists $\theta \in (0, 1)$ such that
	\begin{align}\label{equation_theta}
	\varepsilon < \theta^{B_0} < 1.
	\end{align}

Consider the dynamic updates (\ref{update_compact_w})$\sim$(\ref{update_2c}) of Algorithm \ref{alg_private} and the weight generation rule in Table \ref{table_weight_rule}. If Assumptions \ref{smooth}, \ref{convex}, and \ref{connection} hold, we have the following lemmas. Their proofs are given in Appendix \ref{proof_lemmas}.

%For the distributed optimization problem (\ref{prob_aggr}) with Assumptions \ref{smooth}, \ref{convex}, and \ref{connection}, consider the dynamic updates (\ref{update_compact_w})$\sim$(\ref{update_2c}) of Algorithm \ref{alg_private}. Then, we have the following lemmas. Their proofs are given in Appendix \ref{proof_lemmas}.

%Next, we give four lemmas that will be used to prove Theorem \ref{theorem_Rlinear}. The proofs are given in Appendix \ref{proof_lemmas}.

\begin{lemma}\label{lemma_vr}
For any $\theta \in (0,1)$ and $K = 1, 2, \dots$, we have
	\begin{equation}
	\label{db_0}
	\| \mathbf{v} \|_{\text{F}}^{\theta, K} \le \gamma_1 \| \mathbf{r} \|_{\text{F}}^{\theta, K} + b_1,
	\end{equation}
where $\gamma_1 = \hat{L} (1+\theta^{-1})$ and $b_1 = \theta^{-1} \| \mathbf{v}(1) \|_{\text{F}}$.
%	\begin{equation}\label{equation_gamma_1}
%	\gamma_1 = \hat{L} (1+\theta^{-1}), b_1 = \theta^{-1} \| \mathbf{v}(1) \|_{\text{F}}.
%	\end{equation}
\end{lemma}

\begin{lemma}\label{lemma_uv}
Let $B_0, \varepsilon$ be given by (\ref{equation_varepsilon}) and $\theta$ satisfy (\ref{equation_theta}). For any $K = 1, 2, \dots$, we have
	\begin{equation}\label{ud_0}
	\begin{aligned}
	\| \check{\mathbf{u}} \|_{\text{F}}^{\theta, K} \le \gamma_2 \| \mathbf{v} \|_{\text{F}}^{\theta, K} + b_2,
	\end{aligned}
	\end{equation}
where $\gamma_2 = \frac{\epsilon \| W^{-1} \|_{\max}^{1} \theta (1 - \theta^{B_0})}{(\theta^{B_0} - \varepsilon)(1 - \theta)}$, $\| W^{-1} \|_{\max}^{1} = \sup_{k\ge 1} \| W(k)^{-1} \|_{\max} \le 1/(c_0)^{m B}$, and $b_2 = \frac{\theta^{B_0}}{\theta^{B_0} - \varepsilon} \sum_{i=1}^{B_0} \theta^{-i} \| \check{\mathbf{u}}(i) \|_{\text{F}}$.
%$\gamma_2 = \frac{\epsilon \| W^{-1} \|_{\max}^{1} \theta (1 - \theta^{B_0})}{(\theta^{B_0} - \varepsilon)(1 - \theta)}$
%	\begin{equation}\label{equation_gamma_2}
%	\gamma_2 = \frac{\epsilon \| W^{-1} \|_{\max}^{1} \theta (1 - \theta^{B_0})}{(\theta^{B_0} - \varepsilon)(1 - \theta)}, b_2 = \frac{\theta^{B_0}}{\theta^{B_0} - \varepsilon} \sum_{i=1}^{B_0} \theta^{-i} \| \check{\mathbf{u}}(i) \|_{\text{F}}
%	\end{equation}
%and $\| W^{-1} \|_{\max}^{1} = \sup_{k\ge 1} \| W(k)^{-1} \|_{\max} \le 1/(c_0)^{m B}$.
\end{lemma}

%\begin{lemma}\label{lemma_uv}
%Let $B_{0}$ be a positive integer such that $B_0 \ge B$ and $\varepsilon = \epsilon (1 - \sigma^{m B})^{(B_0 - 1)/(m B)} < 1$ in which $\epsilon = 2m (1+\sigma^{-m B}) / (1 - \sigma^{m B})$ and $ \sigma = (c_0)^{2+m B}$. Let $\theta$ be such that $\varepsilon < \theta^{B_0} < 1$. Then, under Assumption \ref{connection} and the weight generation rule in Table \ref{table_weight_rule}, we have for any $K = 1, 2, \dots$
%	\begin{equation}\label{ud_0}
%	\begin{aligned}
%	\| \check{\mathbf{u}} \|_{\text{F}}^{\theta, K} \le &\ \frac{\epsilon \| W^{-1} \|_{\max}^{1} \theta (1 - \theta^{B_0})}{(\theta^{B_0} - \varepsilon)(1 - \theta)} \| \mathbf{v} \|_{\text{F}}^{\theta, K} \\
%	& + \frac{\theta^{B_0}}{\theta^{B_0} - \varepsilon} \sum_{i=1}^{B_0} \theta^{-i} \| \check{\mathbf{u}}(i) \|_{\text{F}},
%	\end{aligned}
%	\end{equation}
%where $\| W^{-1} \|_{\max}^{1} = \sup_{k\ge 1} \| W(k)^{-1} \|_{\max} \le 1/(c_0)^{m B}$.
%\end{lemma}

\begin{lemma}\label{lemma_xu}
With the same assumptions of Lemma \ref{lemma_uv}, for any $K = 1, 2, \dots$, we have
	\begin{equation}\label{xu_0}
	\begin{aligned}
	\| \check{\mathbf{x}} \|_{\text{F}}^{\theta, K} \le \gamma_3 \| \check{\mathbf{u}} \|_{\text{F}}^{\theta, K} + b_3,
	\end{aligned}
	\end{equation}
where $\gamma_3 = \frac{\eta}{\theta^{B_0} - \varepsilon}\Big( \varepsilon + \frac{\epsilon (1 - \theta^{B_0 - 1})}{1 - \theta} \Big)$ and $b_3 = \frac{\theta^{B_0}}{\theta^{B_0} - \varepsilon}\sum_{i = 1}^{B_0} \theta^{-i}\| \check{\mathbf{x}}(i) \|_{\text{F}}$.
%	\begin{equation}\label{equation_gamma_3}
%	\begin{aligned}
%	&\gamma_3 = \frac{\eta}{\theta^{B_0} - \varepsilon}\Big( \varepsilon + \frac{\epsilon (1 - \theta^{B_0 - 1})}{1 - \theta} \Big), \\
%	&b_3 = \frac{\theta^{B_0}}{\theta^{B_0} - \varepsilon}\sum_{i = 1}^{B_0} \theta^{-i}\| \check{\mathbf{x}}(i) \|_{\text{F}}.
%	\end{aligned}
%	\end{equation}
\end{lemma}

%\begin{lemma}\label{lemma_xu}
%Let $\theta$ be such that $\varepsilon < \theta^{B_0} < 1$ in which $B_0$ and $\varepsilon$ are denoted in Lemma \ref{lemma_uv}. Then, under Assumption \ref{connection} and the weight generation rule in Table \ref{table_weight_rule}, we have for any $K = 1, 2, \dots$
%	\begin{equation}\label{xu_0}
%	\begin{aligned}
%	\| \check{\mathbf{x}} \|_{\text{F}}^{\theta, K} \le & \frac{\eta}{\theta^{B_0} - \varepsilon}\Big( \varepsilon + \frac{\epsilon (1 - \theta^{B_0 - 1})}{1 - \theta} \Big) \| \check{\mathbf{u}} \|_{\text{F}}^{\theta, K}\\
%	& + \frac{\theta^{B_0}}{\theta^{B_0} - \varepsilon}\sum_{i = 1}^{B_0} \theta^{-i}\| \check{\mathbf{x}}(i) \|_{\text{F}},
%	\end{aligned}
%	\end{equation}
%where $\epsilon$ is the same as in Lemma \ref{lemma_uv}.
%\end{lemma}

\begin{lemma}\label{lemma_rx}
Suppose that $\sqrt{1 - \frac{\alpha \eta \bar{\mu}}{\alpha+1}} \le \theta < 1$ and $\eta \le 1/((1+\beta)\bar{L})$ where $\alpha > 0$ and $\beta > 0$. For any $K = 1, 2, \dots$, we have
	\begin{equation}\label{bx_0}
	\begin{aligned}
	\| \mathbf{r} \|_{\text{F}}^{\theta, K} \le \gamma_4 \| \check{\mathbf{x}} \|_{\text{F}}^{\theta, K} + b_4,
	\end{aligned}
	\end{equation}
where $\gamma_4 = (1+\sqrt{m}) \left( 1 + \frac{\sqrt{m}}{\theta} \sqrt{ \frac{\hat{L}(1+\beta) + \alpha \beta \hat{\mu}}{\bar{\mu} \beta}}\right)$ and $b_4 = 2\sqrt{m}\| \bar{y}(1) - x^* \|_{2}$.
%	\begin{equation}\label{equation_gamma_4}
%	\begin{aligned}
%	&\gamma_4 = (1+\sqrt{m}) \left( 1 + \frac{\sqrt{m}}{\theta} \sqrt{ \frac{\hat{L}(1+\beta) + \alpha \beta \hat{\mu}}{\bar{\mu} \beta}}\right),\\
%	&b_4 = 2\sqrt{m}\| \bar{y}(1) - x^* \|_{2}.
%	\end{aligned}
%	\end{equation}
\end{lemma}
%where $\bar{y}(1) = \frac{1}{m}\sum_{i = 1}^{m}y_i(1)$.

\subsection{Convergence Result}
We present the main convergence result of Algorithm \ref{alg_private} in Theorem \ref{theorem_Rlinear}.

\begin{theorem}\label{theorem_Rlinear}
Consider the distributed optimization problem (\ref{prob_aggr}) with Assumptions \ref{smooth}, \ref{convex}, \ref{connection} and the weight generation rule in Table \ref{table_weight_rule}. Denote the constants $C_1$ and $C_2$ as $C_1 = 2\sqrt{\frac{(1+\beta)m\hat{L}}{\beta\bar{\mu}} + \frac{\alpha m\hat{\mu}}{\bar{\mu}}}$ and $C_2 = 2 B_0 \kappa \epsilon \| W^{-1} \|_{\max}^{1} (\varepsilon+\epsilon(B_0 - 1)) (1+\sqrt{m}) (1+\alpha^{-1}) (1 + C_1)$. If the fixed step-size $\eta$ satisfies
\begin{equation}
\eta \in \left(0, \min\left\{\frac{(1 + \alpha^{-1}) (1 - \varepsilon)^2}{\bar{\mu} C_2}, \frac{1}{(1+\beta)\bar{L}}\right\}\right),
\end{equation}
then Algorithm \ref{alg_private} converges to the optimal solution $x^*$ at the linear rate $\mathcal{O}(\theta^k)$ with $\theta_0 \le \theta <1$ in which $\theta_0 = \sqrt[B_0]{\frac{\varepsilon + \sqrt{C_2(C_2 - \varepsilon^2 + 1)}}{1+C_2}}$.
\end{theorem}
%\begin{theorem}\label{theorem_Rlinear}
%\textcolor{blue}{
%Consider the distributed optimization problem (\ref{prob_aggr}) with Assumptions \ref{smooth}, \ref{convex}, \ref{connection} and the weight generation rule in Table \ref{table_weight_rule}. Let $B_{0}$ be a positive integer such that $B_{0}\ge B$ and $\varepsilon = \epsilon (1 - \sigma^{m B})^{(B_0 - 1)/(m B)} < 1$. Denote the constants $C_1$ and $C_2$ as $C_1 = 2\sqrt{\frac{(1+\beta)m\hat{L}}{\beta\bar{\mu}} + \frac{\alpha m\hat{\mu}}{\bar{\mu}}}$ and $C_2 = 2 B_0 \kappa \epsilon \| W^{-1} \|_{\max}^{1} (\varepsilon+\epsilon(B_0 - 1)) (1+\sqrt{m}) (1+\alpha^{-1}) (1 + C_1)$ where $\alpha > 0$, $\beta > 0$, and $\kappa = \hat{L}/\bar{\mu}$. If the fixed step-size $\eta$ satisfies
%\begin{equation}
%\eta \in \left(0, \min\left\{\frac{(1 + \alpha^{-1}) (1 - \varepsilon)^2}{\bar{\mu} C_2}, \frac{1}{(1+\beta)\bar{L}}\right\}\right),
%\end{equation}
%then Algorithm \ref{alg_private} converges to the optimal solution $x^*$ at the linear rate $\mathcal{O}(\theta^k)$ with $\theta_0 \le \theta <1$ in which $\theta_0 = \sqrt[B_0]{\frac{\varepsilon + \sqrt{C_2(C_2 - \varepsilon^2 + 1)}}{1+C_2}}$.
%}
%\end{theorem}

\begin{proof}
From Lemmas \ref{lemma_vr}, \ref{lemma_uv}, \ref{lemma_xu}, and \ref{lemma_rx}, we can obtain the gain constants $\gamma_1$, $\gamma_2$, $\gamma_3$, and $\gamma_4$. We first show that there exists parameters $\alpha$, $\beta$, $\theta$, and $\eta$ such that $\gamma_1\gamma_2\gamma_3\gamma_4 < 1$ where $\alpha > 0$, $\beta > 0$, $\varepsilon < \theta^{B_0} < 1$, $\sqrt{1 - \frac{\alpha\eta\bar{\mu}}{1+\alpha}} \le \theta < 1$, and $\eta \le \frac{1}{(1+\beta)\bar{L}}$. 

%where parameters $\alpha$, $\beta$, $\theta$, and $\eta$ need to satisfy $\alpha > 0$, $\beta > 0$, $\varepsilon < \theta^{B_0} < 1$, $\sqrt{1 - \frac{\alpha\eta\bar{\mu}}{1+\alpha}} \le \theta < 1$, and $\eta \le \frac{1}{(1+\beta)\bar{L}}$.
%in (\ref{equation_gamma_1}), (\ref{equation_gamma_2}), (\ref{equation_gamma_3}), and (\ref{equation_gamma_4})
%	\begin{align*}
%	&\gamma_1 = \hat{L} (1+\theta^{-1}), \\
%	&\gamma_2 = \frac{\epsilon \| W^{-1} \|_{\max}^{1} \theta (1 - \theta^{B_0})}{(\theta^{B_0} - \varepsilon)(1 - \theta)}, \\
%	&\gamma_3 = \frac{\eta}{\theta^{B_0} - \varepsilon}\Big( \varepsilon + \frac{\epsilon (1 - \theta^{B_0 - 1})}{1 - \theta} \Big), \\
%	&\gamma_4 = (1+\sqrt{m}) \left( 1 + \frac{\sqrt{m}}{\theta} \sqrt{ \frac{\hat{L}(1+\beta) + \alpha \beta \hat{\mu}}{\bar{\mu} \beta}}\right),
%	\end{align*}
%where parameters $\alpha$, $\beta$, $\theta$, and $\eta$ need to satisfy $\alpha > 0$, $\beta > 0$, $\varepsilon < \theta^{B_0} < 1$, $\sqrt{1 - \frac{\alpha\eta\bar{\mu}}{1+\alpha}} \le \theta < 1$, and $\eta \le \frac{1}{(1+\beta)\bar{L}}$.
%	\begin{gather*}
%	\alpha > 0, \beta > 0,\\
%	\varepsilon < \theta^{B_0} < 1,\\
%	\sqrt{1 - \frac{\alpha\eta\bar{\mu}}{1+\alpha}} \le \theta < 1, \\
%	\eta \le \frac{1}{(1+\beta)\bar{L}}.
%	\end{gather*}

%We first show that there exists parameters $\alpha$, $\beta$, $\theta$, and $\eta$ such that $\gamma_1\gamma_2\gamma_3\gamma_4 < 1$. If we choose $0.5\le \theta < 1$ and use $\frac{1-\theta^{B_0}}{1-\theta} \le B_0$, it follows that
Due to $0 < \theta < 1$ and $B_{0} \ge 1$, we have 
\begin{equation}\label{equation_theta_B0}
\frac{1-\theta^{B_0}}{1-\theta} \le B_0,
\end{equation}
the proof of which is given in Appendix \ref{proof_theta_B0}. If we choose $0.5\le \theta < 1$ and use (\ref{equation_theta_B0}), it follows that
\begin{equation}
	\begin{aligned}
&\gamma_1\gamma_2\gamma_3\gamma_4 < \\
&\frac{2 \eta B_0 \hat{L}\epsilon\| W^{-1} \|_{\max}^{1} (\varepsilon + \epsilon(B_0-1))(1+\sqrt{m})\left(1 + C_1\right)}{(\theta^{B_0} - \varepsilon)^2},
	\end{aligned}
\end{equation}
where $C_1 = 2\sqrt{\frac{(1+\beta)m\hat{L}}{\beta\bar{\mu}} + \frac{\alpha m\hat{\mu}}{\bar{\mu}}}$. It means that
	\begin{align}\label{equation_eat_right}
	\eta \le \frac{(1 + \alpha^{-1})(\theta^{B_0} - \varepsilon)^2}{\bar{\mu} C_2},
	\end{align}
where $C_2 = 2 B_0 \kappa \epsilon \| W^{-1} \|_{\max}^{1} (\varepsilon+\epsilon(B_0 - 1)) (1+\sqrt{m}) (1+\alpha^{-1}) (1 + C_1)$. In addition, in order to satisfy $\sqrt{1 - \frac{\alpha\eta\bar{\mu}}{1+\alpha}} \le \theta < 1$, it suffices to ensure that $\eta \le (1+\alpha^{-1})/\bar{\mu}$ and
	\begin{align}\label{equation_eat_left}
	\eta \ge \frac{(1 + \alpha^{-1}) (1 - \theta^2)}{\bar{\mu}}.
	\end{align}
Based on (\ref{equation_eat_right}) and (\ref{equation_eat_left}), the key is to verify that the interval of the step-size $\eta$ is not empty. For the convenience of analysis, we shrink the bound in (\ref{equation_eat_left}). That is, we need to make sure that
	\begin{align}\label{equation_interval}
	\eta \in \left[ \frac{(1 + \alpha^{-1}) (1 - \theta^{2 B_0})}{\bar{\mu}}, \frac{(1 + \alpha^{-1})(\theta^{B_0} - \varepsilon)^2}{\bar{\mu} C_2} \right] \not= \emptyset.
	\end{align}
In the following, we show that there exists $\theta$ such that (\ref{equation_interval}) holds. We know that $\varepsilon < \theta^{B_0} < 1$. It can be found that as $\theta$ varies from $\sqrt[B_0]{\varepsilon}$ to 1, the left-bound in (\ref{equation_interval}) monotonically decreases from $\frac{(1 + \alpha^{-1})(1 - \varepsilon^2)}{\bar{\mu}}$ to 0 and the right-bound monotonically increases from 0 to $\frac{(1 + \alpha^{-1}) (1 - \varepsilon)^2}{\bar{\mu} C_2}$. This means that there exists a critical $\theta_0$ such that the left-bound is equal to the right-bound where
\begin{equation}\label{equation_theta_0}
\theta_0 = \sqrt[B_0]{\frac{\varepsilon + \sqrt{C_2(C_2 - \varepsilon^2 + 1)}}{1+C_2}}.
\end{equation}
When $\theta \in [\theta_0, 1)$, (\ref{equation_interval}) holds. Thus, we have $\gamma_1\gamma_2\gamma_3\gamma_4 < 1$. In addition, the step-size $\eta$ can be taken as $\eta \in \left(0, \min\left\{\frac{(1 + \alpha^{-1}) (1 - \varepsilon)^2}{\bar{\mu} C_2}, \frac{1}{(1+\beta)\bar{L}}\right\}\right)$.

%(\ref{bx_0}), (\ref{xu_0}), (\ref{ud_0}), and (\ref{db_0})
Then, combining inequalities (\ref{db_0}), (\ref{ud_0}), (\ref{xu_0}), and (\ref{bx_0}) and based on $\gamma_1\gamma_2\gamma_3\gamma_4 < 1$, it follows that $\| \mathbf{r} \|_{\text{F}}^{\theta, K} \le C_3$ for any $K \ge 1$ where $C_3 = \frac{b_1\gamma_2\gamma_3\gamma_4 + b_2\gamma_3\gamma_4 + b_3\gamma_4 + b_4}{1 - \gamma_1\gamma_2\gamma_3\gamma_4}$
%	\begin{align*}
%	C_3 = \frac{b_1\gamma_2\gamma_3\gamma_4 + b_2\gamma_3\gamma_4 + b_3\gamma_4 + b_4}{1 - \gamma_1\gamma_2\gamma_3\gamma_4}, 
%	\end{align*}
%	\begin{align*}
%	&C_3 = \frac{b_1\gamma_2\gamma_3\gamma_4 + b_2\gamma_3\gamma_4 + b_3\gamma_4 + b_4}{1 - \gamma_1\gamma_2\gamma_3\gamma_4}, \\
%	&b_1 =  \theta^{-1} \| \mathbf{v}(1) \|_{\text{F}}, \ b_2 = \frac{\theta^{B_0}}{\theta^{B_0} - \varepsilon} \sum_{i=1}^{B_0} \theta^{-i} \| \check{\mathbf{u}}(i) \|_{\text{F}}, \\
%	&b_3 = \frac{\theta^{B_0}}{\theta^{B_0} - \varepsilon}\sum_{i = 1}^{B_0} \theta^{-i}\| \check{\mathbf{x}}(i) \|_{\text{F}}, \ b_4 =2\sqrt{m}\| \bar{y}(1) - x^* \|_{2},
%	\end{align*}
is independent of $K$. When $K \to \infty$, $\| \mathbf{r} \|_{\text{F}}^{\theta, K}$ is still bounded by $C_3$. Thus, we have $\| \mathbf{r}(k) \|_{\text{F}} \le C_3 \theta^k$ for all $k \ge 1$ based on the definition of $\| \cdot \|_{\text{F}}^{\theta, K}$, that is, $\| \mathbf{x}(k) - \mathbf{x}^* \|_{\text{F}} \le C_3 \theta^k, \forall k \ge 1$.
\end{proof}
%Using the small gain theorem \cite[Theorem 3.7, Lemma 3.8]{nedic2017achieving}, we can draw the conclusion in Theorem \ref{theorem_Rlinear}.
\begin{remark}
Algorithm \ref{alg_private} involves AES's encryption and decryption operations. These operations do not interfere with the agents' normal updates and thus do not affect the convergence rate and accuracy of Algorithm \ref{alg_private}. In contrast, existing methods may experience accuracy loss due to using the Paillier cryptosystem. This is because the Paillier cryptosystem is designed to handle unsigned integers, requiring fixed-point encoding for processing floating-point numbers.
%Nevertheless, since the Paillier cryptosystem works for unsigned integers, these operations will produce quantization errors when dealing with signed real numbers, which in turn affects the optimization accuracy of Algorithm \ref{alg_private}. However, this influence is limited. This is because the quantization error here is no different from the conventional quantization error that occurs when any algorithm is implemented on a computer \cite{zhang2019admm}. In addition, the quantization error can be greatly reduced by appropriately increasing the encoding precision.
\end{remark}
%When using the Paillier cryptosystem to deal with real numbers, the precision of encoding will affect the convergence accuracy of Algorithm \ref{alg_private}. But this effect is no different from the quantization error that occurs when any algorithm is implemented on a computer. The quantization error can be greatly reduced by appropriately improving the encoding precision. Therefore, we ignore the effect of the cryptosystem on the convergence accuracy of the algorithm.

\section{Privacy Analysis}\label{section_privacy_analysis}
In this section, we analyze how Algorithm \ref{alg_private} performs in terms of privacy preservation under the honest-but-curious adversary and the external eavesdropper (see Definition \ref{definition_honest_attack} and \ref{definition_external_attack} for definitions of these two attackers). We regard each agent $i$'s gradients as its sensitive information \cite{zhu2019deep, gao2023dynamics}. Each attacker can use both its exploitable information and the updates (\ref{init_upda}) to infer agent $i$'s sensitive information. The exploitable information to an honest-but-curious adversary $j$ at iteration $k$ includes the data locally generated by adversary $j$ (such as $y_j(k)$, $w_j(k)$, $x_j(k)$, $s_j(k)$, $\nabla f_j(x_j(k))$, and $a_{lj}(k), l\in\mathcal{N}_j^{\text{out}}(k) \cup \{j\}$) and the information received from its in-neighbor $i$ (such as $a_{ji}(k)y_i(k)$, $a_{ji}(k)s_i(k)$, $a_{ji}(k)w_i(k), i\in\mathcal{N}_j^{\text{in}}(k)$).
%\begin{equation}
%\mathcal{I}_j =\left\{ pk, sk \right\} \cup \left\{ y_j(k), w_j(k), x_j(k), s_j(k), \nabla f_j(x_j(k)), a_{lj}(k)\ |\ l\in\mathcal{N}_j^{\text{out}}(k) \cup \{i\}, k\ge 0  \right\} \cup \left\{ a_{ji}(k)y_i(k), a_{ji}(k)s_i(k), a_{ji}(k)w_i(k)\ |\ i\in\mathcal{N}_j^{\text{in}}(k), k \in \mathcal{K}_{ji} \right\}
%\end{equation}
%where $\mathcal{K}_{ji}$ represents the set of iterations that agent $j$ can receive the information from agent $i$, i.e., the set of iterations that agent $j$ is the out-neighbor of agent $i$. 
For the external eavesdropper, its exploitable information includes all encrypted information transmitted on the communication links from $k=0$ to $K-1$ (such as $[\![ a_{ji}(k)y_i(k) ]\!]$, $[\![ a_{ji}(k)s_i(k) ]\!]$, $[\![ a_{ji}(k)w_i(k) ]\!], i\in\mathcal{V}, j\in\mathcal{N}_i^{\text{out}}(k), k\ge 0$).
%\begin{equation}
%\mathcal{I}_e = \left\{ [\![ a_{ji}(k)y_i(k) ]\!], [\![ a_{ji}(k)s_i(k) ]\!], [\![ a_{ji}(k)w_i(k) ]\!] \ |\ i\in\mathcal{N}, j\in\mathcal{N}_i^{\text{out}}(k), k\ge 0 \right\}.
%\end{equation}
First, we give the definition of privacy preservation in Definition \ref{unobservability}.
% That is to say, the honest-but-curious adversary and the external eavesdropper cannot uniquely infer the gradients of agent $i$ based on their exploitable information and the updates (\ref{init_upda}), i.e., the gradients of agent $i$ are unobservable.
\begin{definition}\label{unobservability}
\cite{rikos2023distributed} (Privacy) For algorithms of solving the distributed optimization problem (\ref{prob_aggr}), the privacy of agent $i$ is preserved if neither the honest-but-curious adversary nor the external eavesdropper can infer its gradient value $\nabla f_i(x_i)$ evaluated at any point $x_i$ and determine a finite range $[a, b]$ (where $a < b$ and $a,b \in \mathbb{R}$) within which its gradient value lies.
\end{definition}

Next, we show that the privacy of agent $i$ can be preserved by the following theorem.

\begin{theorem}\label{theorem_privacy_analysis}
Consider the distributed optimization problem (\ref{prob_aggr}) and the definition of privacy in Definition \ref{unobservability}. The following statements hold:

(a) If agent $i$ has at least one legitimate neighbor at $k = 0$ and $k = 1$, Algorithm \ref{alg_private} can preserve the privacy of agent $i$ and its intermediate states $y_i$, $w_i$, and $s_i$ except at $k=1$.

(b) If agent $i$ has at least one legitimate neighbor only at $k = 0$, Algorithm \ref{alg_private} can still preserve the privacy of agent $i$ but its intermediate states $y_i$, $w_i$, and $s_i$ might be leaked.

(c) If agent $i$ has no legitimate neighbors at any iteration, Algorithm \ref{alg_private} cannot preserve the privacy of agent $i$.
\end{theorem}

%\begin{theorem}\label{theorem_privacy}
%\textcolor{blue}{
%Algorithm \ref{alg_private} can preserve the privacy of agent $i$ against the honest-but-curious adversary and the external eavesdropper if agent $i$ has at least one legitimate neighbor at $k = 0$ and $k = 1$, i.e., $\left| \mathcal{N}_i^{\text{out}}(0) \cup \mathcal{N}_i^{\text{in}}(0) \right| \ge 2$ and $\left| \mathcal{N}_i^{\text{out}}(1) \cup \mathcal{N}_i^{\text{in}}(1) \right| \ge 2$.
%}
%\end{theorem}
%even if the honest-but-curious adversary $j$ is always a out-neighbor of agent $i$ or even the only neighbor when $k>1$, i.e., $j\in\mathcal{N}_i^{\text{out}}(k)$ for $k \ge 0$ and $\mathcal{N}_i^{\text{out}}(k) \cup \mathcal{N}_i^{\text{in}}(k) = \{ j \}$ for $k > 1$.
%The information available to agent $j$ includes its exploitable information and the updates (\ref{init_upda}).
%When agent $i$ has at least one legitimate neighbor at $k = 0$ and $k = 1$, adversary $j$ can be always an out-neighbor of agent $i$ and be the only neighbor at $k > 1$ in the worst case. 
\begin{proof}
(a) We first analyze the inability of the honest-but-curious adversary $j$ to infer the agent $i$'s private information. Here, we consider the worst case that adversary $j$ is always an out-neighbor of agent $i$ and is the only neighbor at $k > 1$, i.e., $j \in \mathcal{N}_i^{\text{out}}(k)$ for $k \ge 0$ and $\mathcal{N}_i^{\text{out}}(k) \cup \mathcal{N}_i^{\text{in}}(k) = \{ j \}$ for $k > 1$. In this case, adversary $j$ can receive the information from agent $i$ at all iterations and obtain the exact updates of agent $i$ at $k > 1$. The information received by adversary $j$ from agent $i$ is as follows:
	\begin{align*}
	&J_y(0) = a_{ji}(0) y_i(0), \\
	&J_s(0) = a_{ji}(0) s_i(0), \\
	&J_w(0) = a_{ji}(0) w_i(0), \\
	&\quad\vdots \\
	&J_y(K-1) = a_{ji}(K-1) y_i(K-1), \\
	&J_s(K-1) = a_{ji}(K-1) s_i(K-1), \\
	&J_w(K-1) = a_{ji}(K-1) w_i(K-1).
	\end{align*}
Among these, although $s_i(0)$ is directly related to the gradient, i.e, $s_i(0) = \nabla f_i(x_i(0))$, it is impossible for adversary $j$ to infer $s_i(0)$ since $a_{ji}(0)$ is randomly selected from $\mathbb{R}$.

%In the following, we analyze that the privacy of agent $i$ cannot be inferred for $k = 0$, $k = 1$, and $k >1$, respectively.
%%to that of $l \in \mathcal{N}_i^{\text{in}}(k)$
%When $k = 0$ and $k=1$, given that $\left| \mathcal{N}_i^{\text{out}}(k) \cup \mathcal{N}_i^{\text{in}}(k) \right| \ge 2$,
Moreover, we show that by combining the above information with the updates (\ref{init_upda}), the honest-but-curious adversary $j$ still cannot infer the agent $i$'s private information. In the updates (\ref{init_upda}), $w_i(k) = 1$ when $k = 1$. Then, adversary $j$ can use $J_y(1), J_s(1), J_w(1)$ to calculate $y_i(1)$ and $s_i(1)$. Given that agent $i$ has at least one legitimate neighbor at $k = 0$ and $k = 1$, there must exist an agent $l \in \mathcal{N}_i^{\text{out}}(k) \cup \mathcal{N}_i^{\text{in}}(k), k = 0, 1$ such that $l \not= j$ holds. We consider the case of $l \in \mathcal{N}_i^{\text{in}}(k), k = 0, 1$, and other cases are similar. At $k = 0$, $s_i(0) = \nabla f_i(x_i(0))$, then, the updates of agent $i$ are
	\begin{align*}
	&y_i(1) = y_i(0) - \eta \nabla f_i(x_i(0)) - J_y(0) + \eta J_s(0)\\ 
	&\qquad\quad + a_{il}(0) \left(y_l(0) - \eta s_l(0) \right), \\
	%&x_i(1) = y_i(1)/(w_i(0) - J_w(0) + a_{il}(0)w_l(0)), \\
	&s_i(1) = - J_s(0) + a_{il}(0)s_l(0) + \nabla f_i(x_i(1)).
	\end{align*}
It can be found that although $y_i(1)$ and $s_i(1)$ are known, the gradients $\nabla f_i(x_i(0))$ and $\nabla f_i(x_i(1))$ of agent $i$ are not inferable due to the fact that $y_i(0)$ and $a_{il}(0)$ are randomly selected from $\mathbb{R}^d$ and $\mathbb{R}$ respectively, and $y_l(0), s_l(0)$ are unknown variables.

At $k = 1$, $w_i(k) = 1$ and $w_l(k) = 1$, then, the updates of agent $i$ are
	\begin{align*}
	&y_i(2) = y_i(1) - \eta s_i(1) - J_y(1) + \eta J_s(1) \\
	&\qquad\quad + a_{il}(1) \left(y_l(1) - \eta s_l(1) \right), \\
	&w_i(2) = 1 - J_w(1) + a_{il}(1), \\
	&s_i(2) = s_i(1) - J_s(1) + a_{il}(1)s_l(1) + \nabla f_i(x_i(2)) \\
	&\qquad\quad - \nabla f_i(x_i(1)).
	\end{align*}
Following a similar line of reasoning of the case of $k=0$, the gradients $\nabla f_i(x_i(1))$ and $\nabla f_i(x_i(2))$ of agent $i$ are still not inferable. In addition, $y_i(2), w_i(2), s_i(2)$ are also not inferable.

At $k>1$, $j \in \mathcal{N}_i^{\text{out}}(k)$ and $\mathcal{N}_i^{\text{out}}(k) \cup \mathcal{N}_i^{\text{in}}(k) = \{ j \}$. Here, we consider the case of $\mathcal{N}_i^{\text{out}}(k) = \{j\}, \mathcal{N}_i^{\text{in}}(k) = \emptyset$ for $k>1$, and other cases are similar. Then, the updates of agent $i$ are
	\begin{align*}
	&y_i(k+1) = y_i(k) - \eta s_i(k) - J_y(k) + \eta J_s(k), \\
	&w_i(k+1) = w_i(k) - J_w(k), \\
	&s_i(k+1) = s_i(k) - J_s(k) + \nabla f_i(x_i(k+1)) - \nabla f_i(x_i(k)).
	\end{align*}
It follows that at each iteration $k$, there are $2d+1$ equations but $6d+2$ unknown variables. As it is consistent, the above system of equations must have infinitely many solutions. Thus, the gradients $\nabla f_i(x_i(k)), \nabla f_i(x_i(k+1))$ and states $y_i(k), w_i(k), s_i(k)$ are not inferable for adversary $j$. Even if $s_i(k+1)=\mathbf{0}$ when Algorithm \ref{alg_private} converges, the gradients of agent $i$ are still not inferable for similar reasoning as above.

In summary, if agent $i$ has at least one legitimate neighbor at $k = 0$ and $k = 1$, the honest-but-curious adversary $j$ cannot infer the agent $i$'s private information and its intermediate states $y_i$, $w_i$, and $s_i$ except at $k=1$. Furthermore, there are infinitely many sets of values for the agent $i$'s private variables in all the above inference analyses. This means that the honest-but-curious adversary $j$ also cannot determine a finite range $[a, b]$ (where $a < b$ and $a,b \in \mathbb{R}$) in which the agent $i$’s gradient value lies.
%$\left| \mathcal{N}_i^{\text{out}}(0) \cup \mathcal{N}_i^{\text{in}}(0) \right| \ge 2$ and $\left| \mathcal{N}_i^{\text{out}}(1) \cup \mathcal{N}_i^{\text{in}}(1) \right| \ge 2$
%$$y_i(k),\ k=0, 2, 3, \cdots, K$$
%$$w_i(k),\ k=0, 2, 3, \cdots, K$$
%$$s_i(k),\ k=0, 2, 3, \cdots, K-1$$
%Note that when $k = K-1$, $y_i(K)$ and $y_j(K)$, $w_i(K)$ and $w_j(K)$, $x_i(K)$ and $x_j(K)$, $s_i(K)$ and $s_j(K)$ will reach a consensus. However, because of a similar reasoning to the above, the gradients $\nabla f_i(x_i(K-1))$ and $\nabla f_i(x_i(K))$ of agent $i$ are still not inferable.

Next, we analyze the inability of the external eavesdropper to infer the agent $i$'s private information, which also applies to parts (b) and (c) of Theorem \ref{theorem_privacy_analysis}. The exploitable information of the external eavesdropper includes the information that agent $i$ interacts with all neighbors, however, the information is encrypted by AES. Even if combined with the updates (\ref{init_upda}), the external eavesdropper is still unable to obtain more useful information. Thus, the external eavesdropper can neither infer and determine the agent $i$'s private information.

%When agent $i$ has at least one legitimate neighbor only at $k = 0$, the honest-but-curious adversary $j$ can be always an out-neighbor of agent $i$ and be the only neighbor at $k \ge 1$ in the worst case, i.e., $j\in\mathcal{N}_i^{\text{out}}(k)$ for $k \ge 0$ and $\mathcal{N}_i^{\text{out}}(k) \cup \mathcal{N}_i^{\text{in}}(k) = \{ j \}$ for $k \ge 1$.
(b) We consider the worst case that the honest-but-curious adversary $j$ is always an out-neighbor of agent $i$ and is the only neighbor at $k \ge 1$, i.e., $j\in\mathcal{N}_i^{\text{out}}(k)$ for $k \ge 0$ and $\mathcal{N}_i^{\text{out}}(k) \cup \mathcal{N}_i^{\text{in}}(k) = \{ j \}$ for $k \ge 1$. We here consider the case of $\mathcal{N}_i^{\text{out}}(k) = \{j\}, \mathcal{N}_i^{\text{in}}(k) = \emptyset$ for $k \ge 1$, and other cases are similar. Then, the update of $w_i(k)$ of agent $i$ for $k \ge 1$ is
%the case of $\mathcal{N}_i^{\text{out}}(k) = \{j\}, \mathcal{N}_i^{\text{in}}(k) = \{j\}$ is similar.
%If adversary $j$ is always the out-neighbor of agent $i$ and is the only neighbor at $k \ge 1$, i.e., $j\in\mathcal{N}_i^{\text{out}}(k)$ for $k \ge 0$ and $\mathcal{N}_i^{\text{out}}(k) \cup \mathcal{N}_i^{\text{in}}(k) = \{ j \}$ for $k \ge 1$, then the update of $w_i(k)$ of agent $i$ for $k \ge 1$ is (here we consider the case of $\mathcal{N}_i^{\text{out}}(k) = \{j\}, \mathcal{N}_i^{\text{in}}(k) = \emptyset$ for $k \ge 1$, the case of $\mathcal{N}_i^{\text{out}}(k) = \{j\}, \mathcal{N}_i^{\text{in}}(k) = \{j\}$ is similar)
	\begin{equation}\label{corollary_equation_w}
	w_i(k+1) = w_i(k) - J_w(k).
	\end{equation}
Since $w_i(k) = 1$ when $k=1$, adversary $j$ can infer $w_i(k), k\ge 1$ based on $J_w(k), k\ge 1$ and (\ref{corollary_equation_w}). Then, according to the exploitable information $J_y(k), J_s(k), J_w(k), k\ge 1$, adversary $j$ can also infer $y_i(k), s_i(k), k\ge 1$. To infer the agent $i$'s private information, adversary $j$ can establish the following equations:
	\begin{equation}\label{corollary_equation_sk}
	\begin{aligned}
	&s_i(2) = s_i(1) - J_s(1) + \nabla f_i(x_i(2)) - \nabla f_i(x_i(1)),\\
	&\quad\vdots \\
	&s_i(K) = s_i(K-1) - J_s(K-1) + \nabla f_i(x_i(K)) \\
	&\qquad\qquad - \nabla f_i(x_i(K-1)).
	\end{aligned}
	\end{equation}
There are $(K-1)d$ equations but $K d$ unknown variables although $s_i(k)$ and $J_s(k), k\ge1$ are known. Thus, the honest-but-curious adversary $j$ cannot infer the private information of agent $i$ nor determine a finite range in which the agent $i$'s private information lies.

%When agent $i$ has no legitimate neighbors at any iteration, the honest-but-curious adversary $j$ can be always an out-neighbor of agent $i$ and be always the only neighbor in the worst case.
(c) We consider the worst case that the honest-but-curious adversary $j$ is always an out-neighbor of agent $i$ and is always the only neighbor, i.e., $j\in\mathcal{N}_i^{\text{out}}(k)$ and $\mathcal{N}_i^{\text{out}}(k) \cup \mathcal{N}_i^{\text{in}}(k) = \{ j \}$ for $k \ge 0$. Then, the update of $s_i(k)$ of agent $i$ at $k=0$ is
%If the honest-but-curious adversary $j$ is always the out-neighbor of agent $i$ and always the only neighbor, then the update of $s_i(k)$ of agent $i$ at $k=0$ is
	\begin{equation}\label{corollary_equation_s0}
	s_i(1) = - J_s(0) + \nabla f_i(x_i(1)).
	\end{equation}
Adversary $j$ can infer $\nabla f_i(x_i(1))$ based on the known $s_i(1)$ and (\ref{corollary_equation_s0}), and further infer $\nabla f_i(x_i(k)), k>1$ based on (\ref{corollary_equation_sk}).
\end{proof}

%or agent $i$ has other legitimate neighbors at $k>1$
\begin{remark}
In the analysis of Theorem \ref{theorem_privacy_analysis}, it is conservative to consider that the honest-but-curious adversary $j$ is always an out-neighbor of agent $i$ and even the only neighbor. This is because Algorithm \ref{alg_private} applies to time-varying directed graphs. If the communication between agent $i$ and adversary $j$ is time-varying, the information available to adversary $j$ about agent $i$ is less and intermittent, which adds more difficulty for adversary $j$ to infer the agent $i$'s private information. Furthermore, the privacy preservation of Algorithm \ref{alg_private} requires that agent $i$ has at least one legitimate neighbor only at the initial iteration (see part (b) of Theorem \ref{theorem_privacy_analysis}). This condition is much weaker than that in \cite{gao2023dynamics, zhang2019admm}, which requires that agent $i$ has at least one legitimate neighbor at all iterations. Combining parts (a) and (b) of Theorem \ref{theorem_privacy_analysis}, we know that Algorithm \ref{alg_private} can preserve the intermediate states $y_{i}$, $w_{i}$, $s_{i}$ of agent $i$ if it has at least one legitimate neighbor at both $k = 0$ and $k = 1$. That is to say, if these intermediate states need to be preserved, it is sufficient for agent $i$ to have at least one legitimate neighbor in the first two iterations.
\end{remark}

%\begin{corollary}\label{corollary_inter_var}
%If agent $i$ only has at least one legitimate neighbor at $k = 0$, i.e., $|\mathcal{N}_i^{\text{out}}(0)\cup\mathcal{N}_i^{\text{in}}(0)|\ge 2$, $j\in\mathcal{N}_i^{\text{out}}(k)$ for $k \ge 0$ and $\mathcal{N}_i^{\text{out}}(k) \cup \mathcal{N}_i^{\text{in}}(k) = \{ j \}$ for $k \ge 1$, then the privacy of agent $i$ can still be preserved but the intermediate states $y_i(k), w_i(k), s_i(k), k\ge 1$ would be leaked to the honest-but-curious adversary $j$.
%\end{corollary}

%\begin{corollary}\label{corollary_only_neighbor}
%If the honest-but-curious adversary $j$ is always the out-neighbor of agent $i$ and always the only neighbor, i.e., $j\in\mathcal{N}_i^{\text{out}}(k)$ and $\mathcal{N}_i^{\text{out}}(k) \cup \mathcal{N}_i^{\text{in}}(k) = \{ j \}$ for $k \ge 0$, then adversary $j$ can infer the privacy of agent $i$.
%\end{corollary}

\begin{remark}
Suppose there is a set of honest-but-curious adversaries $\mathcal{H}$, all of whom are the agent $i$'s neighbors and can share information with each other, colluding to infer the agent $i$'s private information. Even so, the privacy of agent $i$ is still not inferable if there exists at least one agent belongs to $\mathcal{N}_i^{\text{out}}(0)\cup\mathcal{N}_i^{\text{in}}(0)$ but not to $\mathcal{H}$. This is because the analysis of Theorem \ref{theorem_privacy_analysis} considers the worst-case scenario, which includes the case that a set of honest-but-curious adversaries are in collusion.
\end{remark}

\begin{remark}\label{remark_pushdiging}
The weight generation rule in Table \ref{table_weight_rule} plays an important role in the privacy analysis of Algorithm \ref{alg_private}. For this reason, existing distributed optimization algorithms over directed graphs such as Push-DIGing \cite{nedic2017achieving} and ADD-OPT \cite{xi2018addopt} cannot protect the privacy of agents. We take ADD-OPT as an example for illustration. In ADD-OPT algorithm, $a_{li} = \frac{1}{d_i^{\text{out}}+1}$ if $l\in\mathcal{N}_i^{\text{out}}$ otherwise $a_{li} = 0$ and $w_i(0) = 1$ for all agents $i$. Once the honest-but-curious adversary $j$ knows the out-degree of agent $i$, it can infer $y_i(k), s_i(k), w_i(k)$ for $k\ge 0$ through $J_y(k), J_s(k), J_w(k)$. Then, the honest-but-curious adversary $j$ can infer $\nabla f_i(x_i(k)), \forall k$ based on $s_i(0) = \nabla f_i(x_i(0))$ and agent $i$'s update of $s_i(k)$. Even if all weights are random and time-varying and agent $i$ always has at least one legitimate neighbor, the honest-but-curious adversary $j$ can still infer $\nabla f_i(x_i(0))$ based on $w_i(0) = 1$ and $\nabla f_i(x_i(0)) = s_i(0)$.
\end{remark}
%Without a proper weight generation rule, Algorithm \ref{alg_private} and other distributed optimization algorithms such as Push-DIGing \cite{nedic2017achieving} and ADD-OPT \cite{xi2018addopt} using column stochastic weight matrices cannot protect the privacy of agents.
% Even if all weights are random and time-varying, the honest-but-curious adversary $j$ can still infer $\nabla f_i(x_i(k)), \forall k$ through $w_i(0) = 1$, $s_i(0) = \nabla f_i(x_i(0))$, and agent $i$'s updates of $w_i(k)$ and $s_i(k)$. 

\begin{remark}\label{remark_no_crytptosystem}
If there is no cryptosystem, the external eavesdropper can intercept all information about agent $i$ on the communication links. Then, the external eavesdropper can infer the agent $i$'s private information according to the updates (\ref{init_upda}), following a similar scheme in the proof of Theorem \ref{theorem_privacy_analysis}(c).
\end{remark}
\section{Numerical Simulations}\label{section_numerical_simulation}
In this section, we evaluate the performance of the proposed cryptography-based privacy-preserving distributed algorithm using the canonical sensor fusion problem \cite{xu2018convergence}. In the problem, all $m$ sensors collectively estimate an unknown parameter $x \in \mathbb{R}^d$ to minimize the sum of their local loss functions, which is formulated as
\begin{equation}\label{problem_sensor}
\min_{x\in\mathbb{R}^d} = \sum_{i=1}^{m}\left( \| z_i - M_i x \|^2 + \omega_i \| x \|^2 \right),
\end{equation}
where $M_i \in \mathbb{R}^{s\times d}$ is the measurement matrix, $z_i = M_i x + \xi_i \in \mathbb{R}^{s}$ represents the observation from sensor $i$ with noise $\xi_i$, and $\omega_i$ is the regularization parameter. The measurement matrix $M_i$ is drawn uniformly from $[0, 10]$. The noise $\xi_i$ follows the i.i.d. Gaussian noise with the mean of zero and unit variance. We generate a predefined parameter $\tilde{x}$ that is uniformly drawn from $[0,1]$, and then construct the observation $z_i$ using $z_i = M_i \tilde{x} + \xi_i$. The regularization parameters of all sensors are consistently set to 0.01. We use the PyCryptodome library \cite{PyCryptodome} to implement AES and choose the key size as 256 bits. All numerical simulations are conducted in the Python 3.8.13 environment with an Intel(R) Xeon(R) Bronze 3204 CPU @ 1.90GHz 1.90 GHz, 32GB RAM.
%We use the python-paillier module \cite{PythonPaillier} to implement the Paillier cryptosystem. 

\begin{figure}[!h]
  \centering
  \includegraphics[]{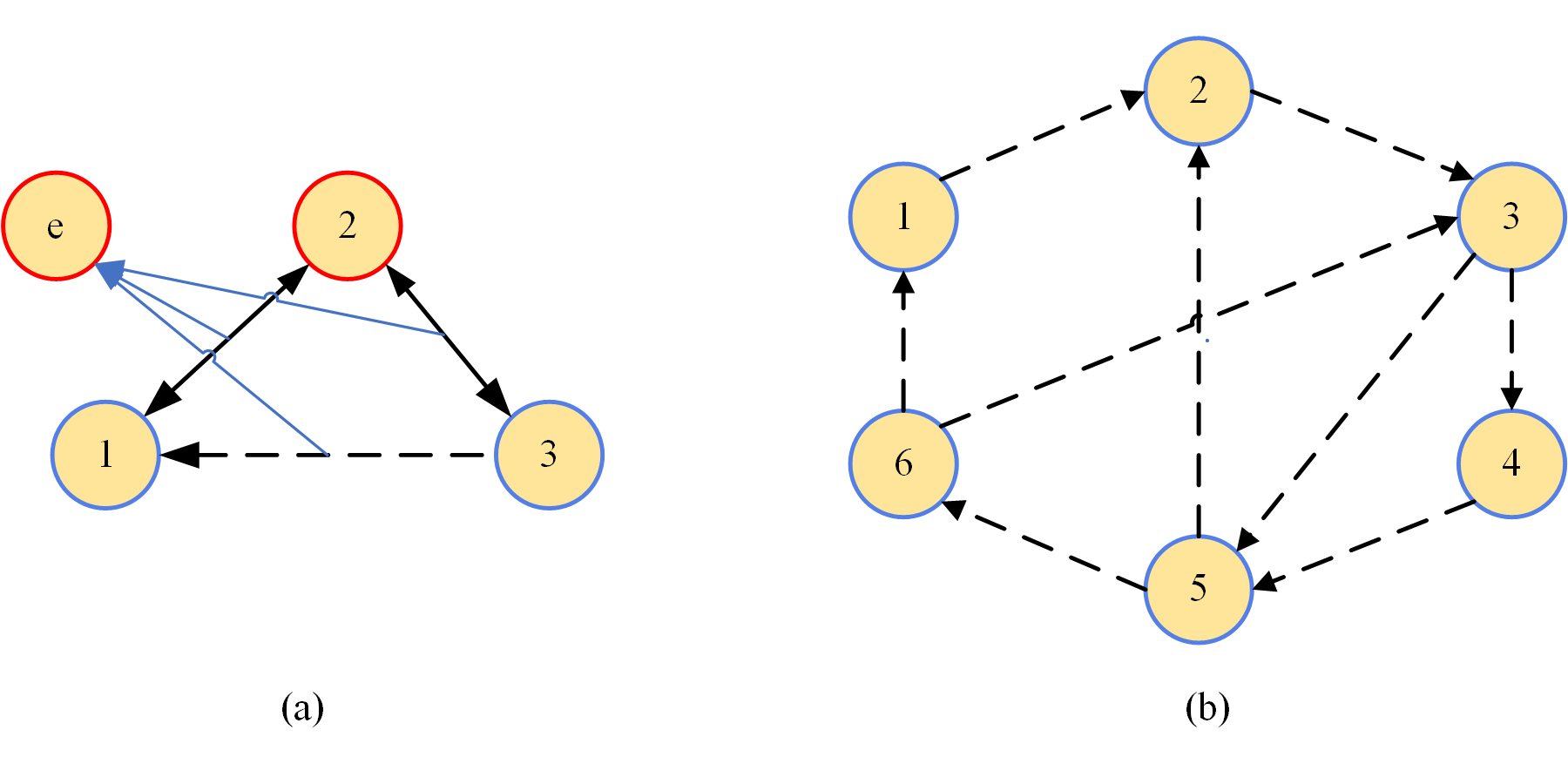}
  \caption{The two communication graphs. A dashed line indicates that the link is not necessarily activated.}
  \label{fig_graph}
\end{figure}

\begin{figure}[!h]
  \centering
  \includegraphics[width = \linewidth]{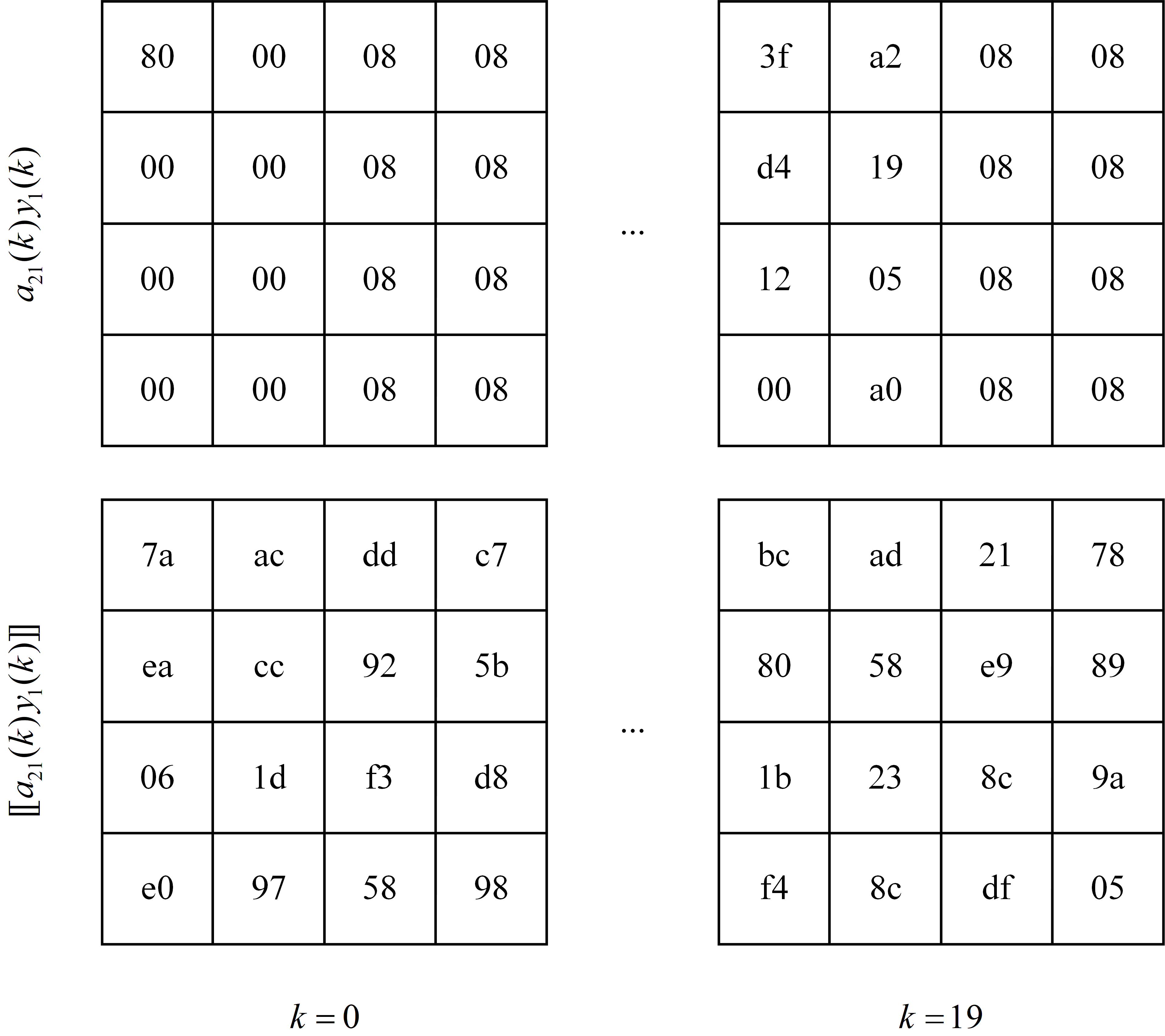}
  \caption{Comparison of the hexadecimal of the true information $a_{21}(k)y_1(k)$ with that of the encrypted information $[\![ a_{21}(k)y_1(k) ]\!]$ stolen by the external eavesdropper.}
  \label{fig_e_privacy}
\end{figure}

\begin{figure}[!h]
  \centering
  \includegraphics[width = \linewidth]{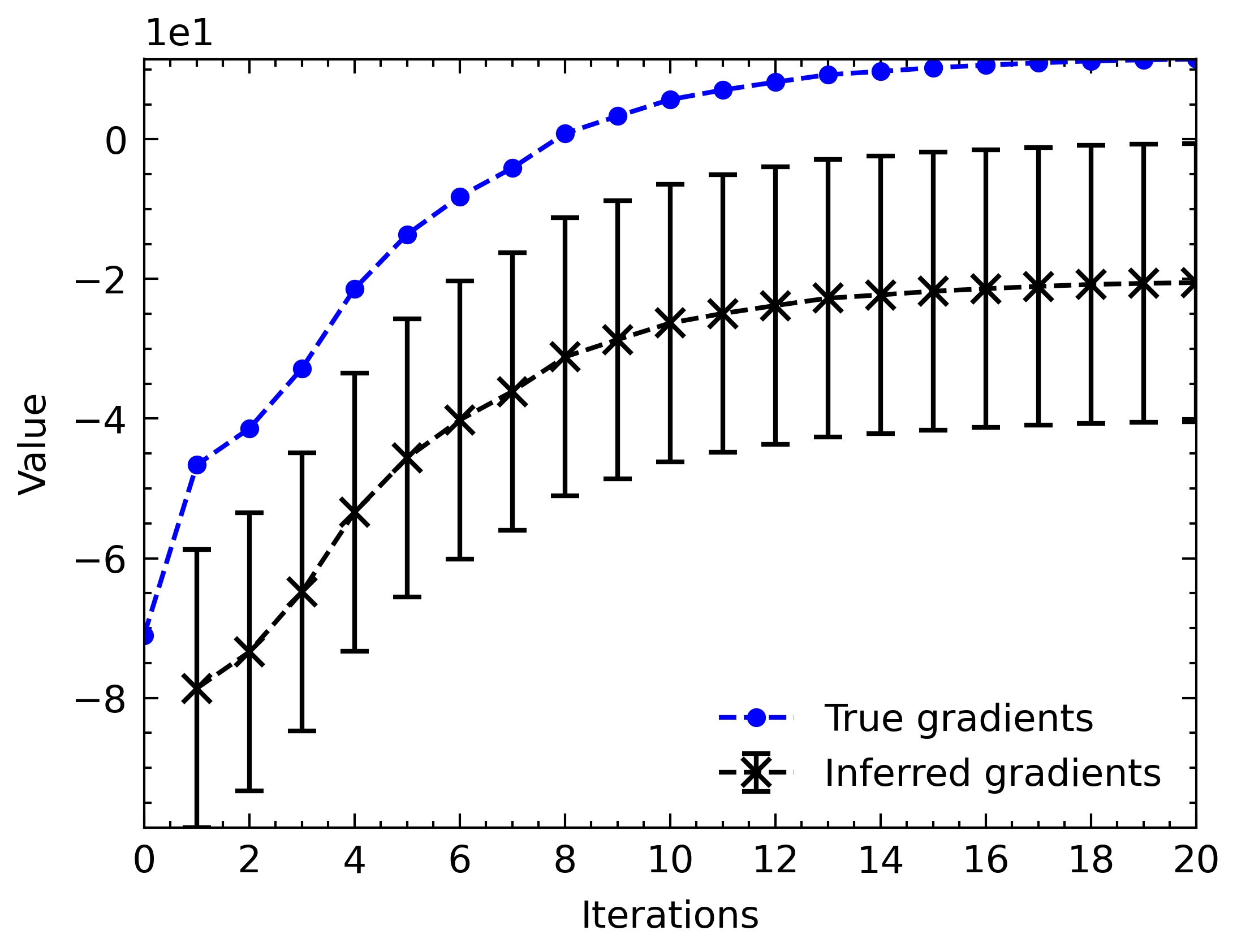}
  \caption{Comparison of agent 1's true gradients with inferred gradients by the honest-but-curious adversary.}
  \label{fig_privacy}
\end{figure}

\begin{figure}[!h]
  \centering
  \includegraphics[width = \linewidth]{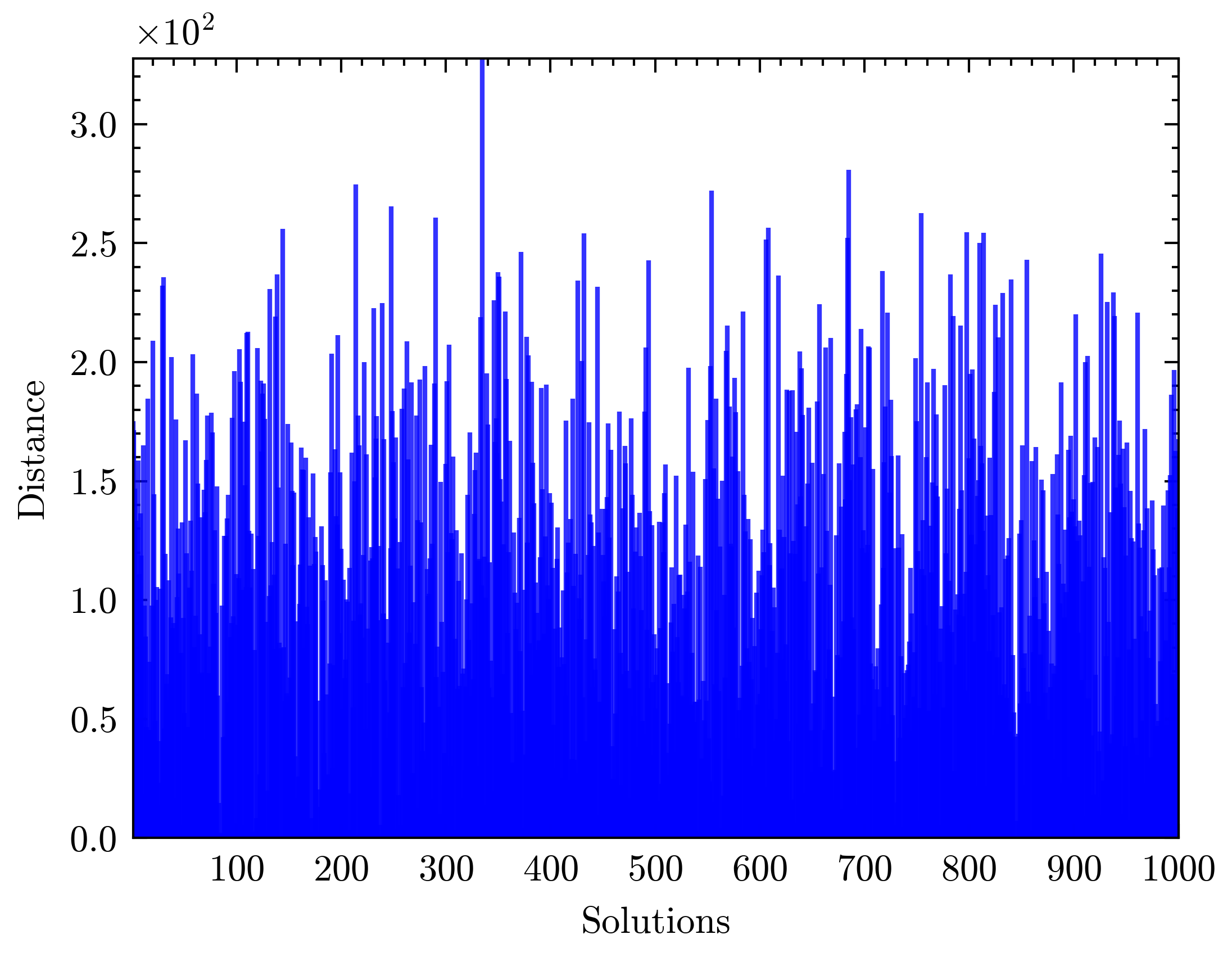}
  \caption{The distance of the inferred gradients by the honest-but-curious adversary to the true gradients.}
  \label{fig_distance}
\end{figure}
%who tries to infer the exact gradients of agent 1
\subsection{Evaluation of Privacy of Algorithm \ref{alg_private}}\label{subsection_privacy}
In this subsection, we evaluate the ability of the proposed privacy-preserving algorithm to preserve exact gradients of agents against the honest-but-curious adversary and the external eavesdropper. We consider the worst case as shown in Fig. \ref{fig_graph}(a). There are three agents 1, 2, 3, and an external eavesdropper $e$ of which agent 2 is the honest-but-curious adversary. Agent 2 is always the neighbor of agent 1, and agent 3 is agent 1's in-neighbor only at the initial iteration, which is consistent with the setup of part (b) of Theorem \ref{theorem_privacy_analysis}. The external eavesdropper $e$ always has access to all information on the communication links. Both agent 2 and external eavesdropper $e$ try to infer the exact gradients of agent 1. For demonstration purposes, we set $s=1$ and $d=1$ in problem (\ref{problem_sensor}) (but other values of $s$ and $d$ can also obtain similar results). In addition, we set the step-size of Algorithm \ref{alg_private} as $5\times 10^{-3}$. Fig. \ref{fig_e_privacy} shows the hexadecimal of the true information $a_{21}(k)y_1(k)$ sent from agent 1 to agent 2, as well as the hexadecimal of the encrypted information $[\![ a_{21}(k)y_1(k) ]\!]$ stolen by external eavesdropper $e$. It can be seen that the true information is disrupted and unrecognizable after being encrypted. Thus, the encrypted information does not reveal any privacy of agent 1 to the external eavesdropper $e$. The blue dots in Fig. \ref{fig_privacy} give exact gradients of agent 1 at $0\le k\le K$. The honest-but-curious agent 2 can try to infer agent 1's exact gradients at $1\le k \le K$ using the system of equations (\ref{corollary_equation_sk}). Theoretically, there are infinitely many solutions to the system of equations. The black crosses in Fig. \ref{fig_privacy} display the mean and variance of 1000 sets of random and special solutions. Clearly, it is almost impossible for agent 2 to infer the exact gradients of agent 1 without additional side information. Note that the inferred gradients by agent 2 at $1\le k\le K$ have the same trend as the exact gradients. This is because when $d = 1$, the system of equations (\ref{corollary_equation_sk}) has $K-1$ equations and $K$ unknown variables. However, once the worst case is no longer satisfied (agent 1 has other neighbors at $k \ge 1$ or agent 2 is not an out-neighbor of agent 1 at a certain iteration), the same trend will not appear. Part (a) of Theorem \ref{theorem_privacy_analysis} has already confirmed this. In addition, we estimate how close the inferred gradients are to the exact gradients by using the relative Euclidean distance, that is, $\sum_{k=1}^{K} \|\nabla \tilde{f}_{1}(x_{1}(k)) - \nabla f_{1}(x_{1}(k))\|_{2}/\|\nabla f_{1}(x_{1}(k))\|_{2}$, where $\nabla \tilde{f}_{1}(x_{1}(k))$ is denoted as the inferred gradient by agent 2 about agent 1 and $\nabla f_{1}(x_{1}(k))$ is the exact gradient of agent 1 at iteration $k$. Fig. \ref{fig_distance} illustrates the distance between those 1000 sets of inferred gradients and the exact gradients. It can be observed that the inferred gradients are far from the exact gradients. In fact, due to the existence of an infinite number of solutions, the distance between the inferred gradients and the exact gradients can be arbitrarily large. Thus, it is also difficult for the honest-but-curious adversary to obtain an approximate estimation of the exact gradients.

%We randomly show five special solutions of the system of equations in Fig. \ref{}, all of which are far from the true gradient values.
%The upper is the true information $a_{21}(k)y_1(k)$ sent by agent 1 to agent 2, and the lower is the encrypted messages $[\![ a_{21}(k)y_1(k) ]\!]$ transmitted on the communication link from agent 1 to agent 2.
%Fig. \ref{fig_e_privacy} shows the true information $a_{21}(k)y_1(k)$ sent by agent 1 to agent 2 and the encrypted information $[\![ a_{21}(k)y_1(k) ]\!]$ stolen by external eavesdropper $e$ at $0\le k\le K-1$.
%Compared with the true information, the encrypted information is haphazard and irregular, and thus does not reveal any privacy of agent 1 (other information about agent 1 would yield the same result).

\begin{figure}[!h]
  \centering
  \includegraphics[width = \linewidth]{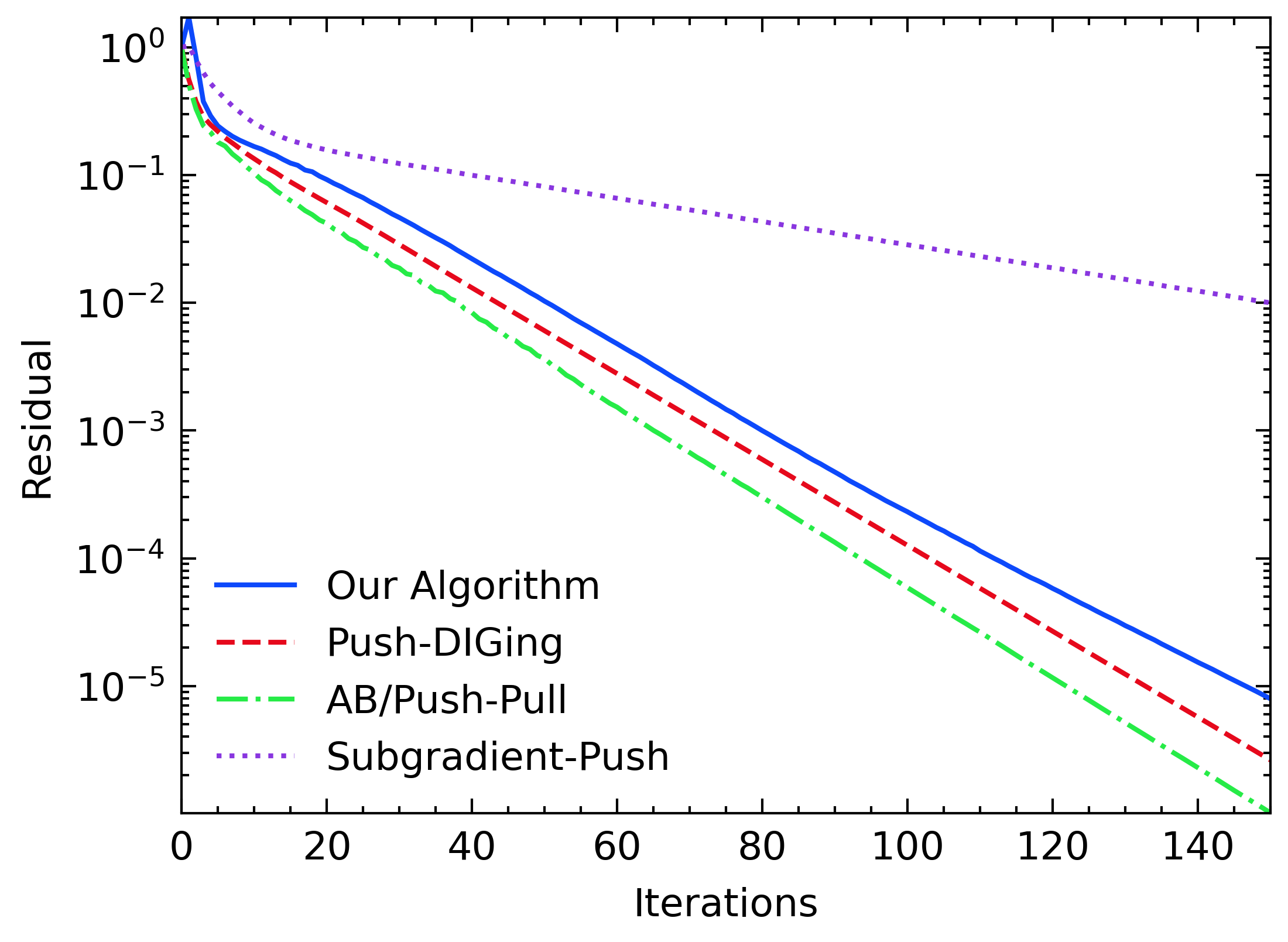}
  \caption{Comparison of the convergence of Algorithm \ref{alg_private} with Push-DIGing \cite{nedic2017achieving}, AB/Push-Pull \cite{nedic2022abpushpull}, and Subgradient-Push \cite{nedic2015distributed} over time-varying directed graphs.}
  \label{fig_tvgraphs}
\end{figure}
\subsection{Evaluation of Convergence of Algorithm \ref{alg_private}}\label{subsection_convergence}
In this subsection, we evaluate the convergence performance of the proposed privacy-preserving algorithm over time-varying directed graphs. We consider a network with six agents as shown in Fig. \ref{fig_graph}(b), and set the dimensions $s = 3$ and $d=2$ of the problem (\ref{problem_sensor}). To construct time-varying directed graphs, we let each directed edge in Fig. \ref{fig_graph}(b) be activated with 90\% probability before each iteration. These activated directed edges form the communication graph of the current iteration. We run our algorithm 100 times with the step-size of $1.1 \times 10^{-3}$ and calculate the average of the relative residual denoted as ${\| \mathbf{x}(k) - \mathbf{x}^*\|^2}/{\| \mathbf{x}(0) - \mathbf{x}^* \|^2}$ per iteration, which is shown by the solid blue line in Fig. \ref{fig_tvgraphs}. For comparison, we also implement the algorithms Push-DIGing \cite{nedic2017achieving} with step-size $1.2\times 10^{-3}$, AB/Push-Pull \cite{nedic2022abpushpull} with step-size $1.2\times 10^{-3}$, and Subgradient-Push \cite{nedic2015distributed} with step-size $1/(k+3\times10^{3})$, all of which are applicable to time-varying directed graphs. The first two algorithms have a linear convergence rate, and the last one has a sub-linear convergence rate. The results are shown by the red, green, and purple lines in Fig. \ref{fig_tvgraphs}, respectively. It can be seen from Fig. \ref{fig_tvgraphs} that the proposed privacy-preserving algorithm has a comparable linear convergence rate over time-varying directed graphs.
%Furthermore, as stated in Remark \ref{remark_convergence}, the weight generation rule only affects the convergence of the algorithm at the initial iteration and does not interfere with the convergence afterward.

\subsection{Comparison with Other Cryptography-Based Privacy-Preserving Distributed Algorithms}\label{subsection_simulation_comparison}
In this subsection, we compare our privacy-preserving algorithm with existing privacy-preserving distributed algorithms that are also based on cryptosystems in terms of both the convergence as well as the running time. The algorithms being used for comparison are the ADMM-based privacy-preserving algorithm proposed in \cite{zhang2019admm} (referred to as PP-ADMM) and the distributed projected subgradient descent-based privacy-preserving algorithm proposed in \cite{zhang2019enabling} (referred to as PP-DPGD). This is because, among the existing cryptography-based privacy-preserving methods, only these two methods address the same optimization problem as in this paper. We use the python-paillier library \cite{PythonPaillier} to implement the Paillier cryptosystem in both algorithms, with a key size of 3072 bits. Both algorithms are only applicable to a fixed undirected graph. Thus, for PP-ADMM and PP-DPGD, we employ a fixed undirected graph by maintaining the edges in Fig. \ref{fig_graph}(b) and removing their directions. For the algorithm proposed in this paper, we continue to use the time-varying directed graph depicted in Fig. \ref{fig_graph}(b). We set the dimensions $s = 3$ and $d = 2$ in problem (\ref{problem_sensor}). We set the step-size of our algorithm as $1.1\times 10^{-3}$, the step-size of PP-DPGD as $1/(1\times10^{4}+k)$, and the parameters of PP-ADMM as $\bar{b} = 0.5, \gamma = 0.1$. In the convergence case, Fig. \ref{fig_convergence_crypesystem} shows the average relative residuals of three algorithms over 100 trials. It is clear that our proposed privacy-preserving algorithm converges much faster than the other two algorithms. This is because both PP-ADMM and PP-DPGD have only a sub-linear convergence rate.
%We all use the python-paillier module \cite{PythonPaillier} to implement the encryption and decryption operations of three algorithms, and choose the key size as 1024-bit.

\begin{figure}[!h]
  \centering
  \includegraphics[width = \linewidth]{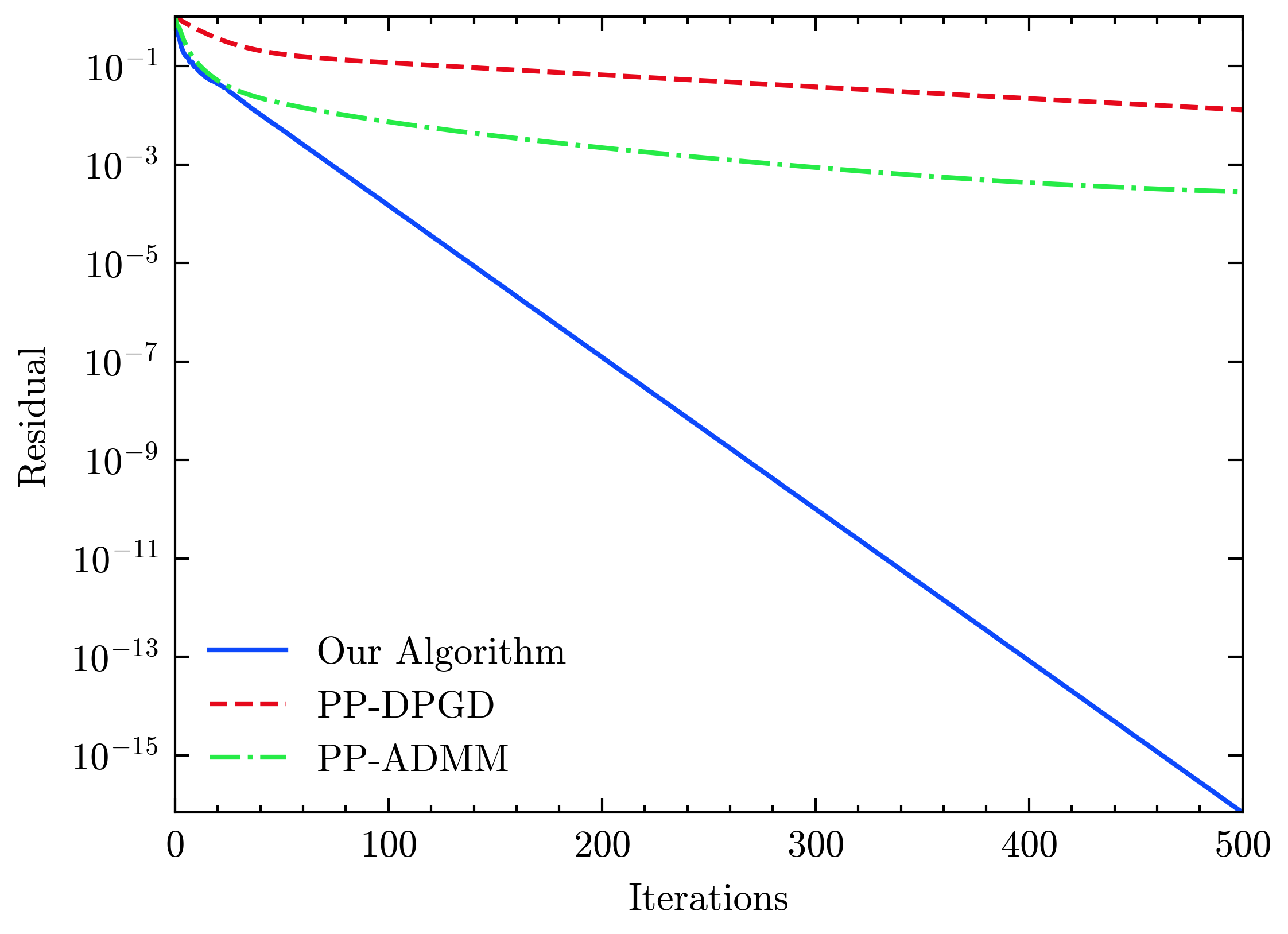}
  \caption{Comparison of the convergence of Algorithm \ref{alg_private} with PP-DPGD \cite{zhang2019enabling} and PP-ADMM \cite{zhang2019admm}.}
  \label{fig_convergence_crypesystem}
\end{figure}

\begin{figure}[!h]
  \centering
  \includegraphics[width = \linewidth]{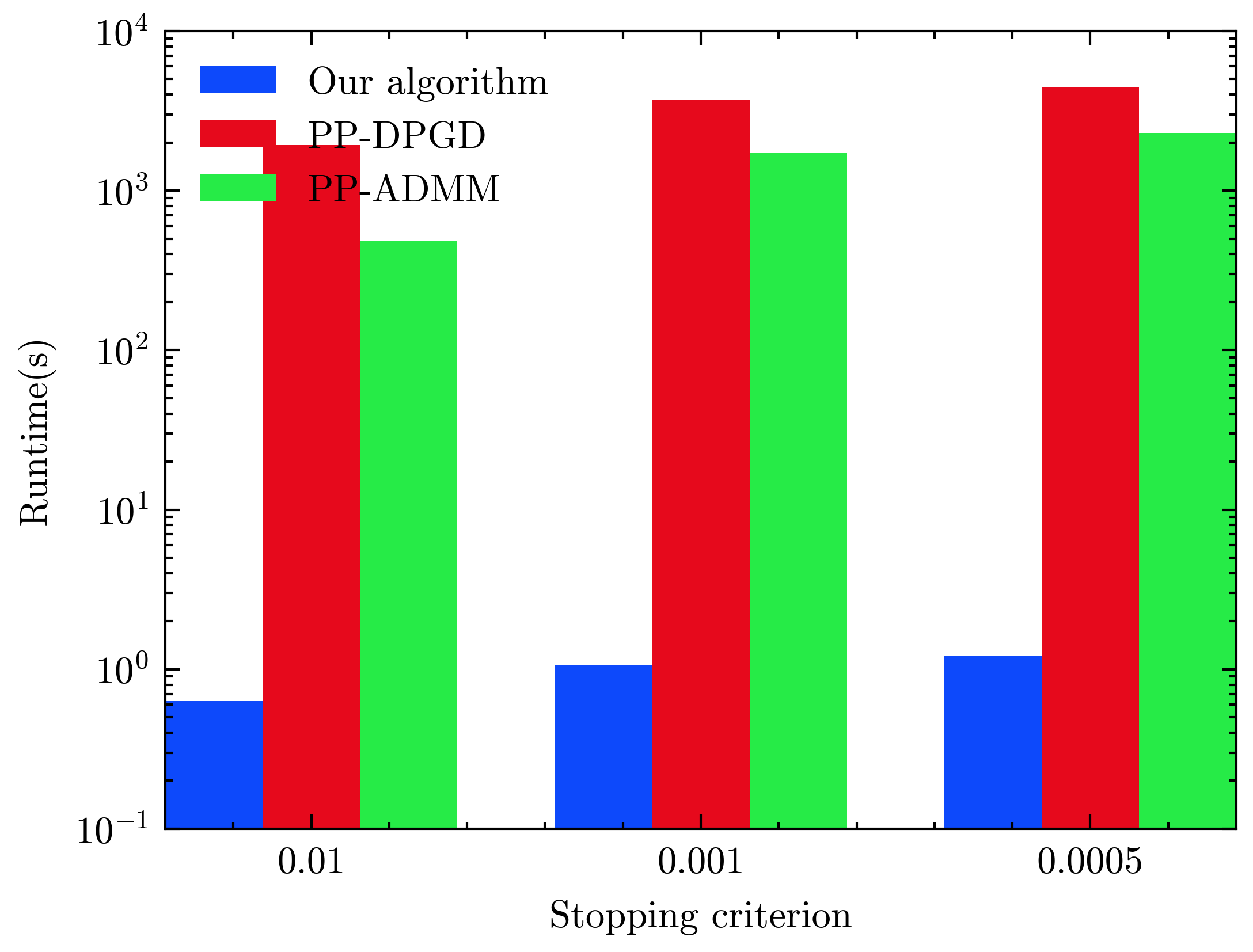}
  \caption{Comparison of the running time of Algorithm \ref{alg_private} with PP-DPGD \cite{zhang2019enabling} and PP-ADMM \cite{zhang2019admm} under stopping criteria 0.01, 0.001, 0.0005.}
  \label{fig_runtime}
\end{figure}
%The dimensions of problem (\ref{problem_sensor}) are set to $s = 3$ and $d=2$. The key size $n\_length$ is chosen as 1024-bit.

In the following, we compare the differences between the three algorithms in terms of the running time. Here, the running time is defined as the total computation time of all agents. To record the running time of the algorithms, we set the same stopping criteria and a maximum number of iterations for the three algorithms. When the relative residual of the algorithm is less than or equal to the given stopping criterion or the iteration count of the algorithm reaches the maximum number of iterations, we count the total number of iterations of the algorithm as well as its running time. Fig. \ref{fig_runtime} depicts the running time of the three algorithms under three different stopping criteria. It is clear that the running time of our proposed privacy-preserving algorithm is significantly less than that of the other two algorithms under all stopping criteria. This is not only due to the faster convergence rate of the proposed privacy-preserving algorithm resulting in fewer iterations (as shown in Fig. \ref{fig_convergence_crypesystem}) but also due to the faster encryption and decryption speed of AES.

To further demonstrate the advantages of the proposed privacy-preserving algorithm in terms of convergence and running time, we count the running time and the number of iterations of three algorithms under lower stopping criterion and higher dimensional data. When dimensions $s = 9$ and $d = 6$, the step-sizes of our algorithm and PP-DPGD are adjusted from $1.1\times 10^{-3}$, $1/(1\times 10^{4}+k)$ to $3\times 10^{-4}$, $1/(2\times 10^4 + k)$, and the parameters of PP-ADMM are changed from $\bar{b} = 0.5, \gamma = 0.1$ to $\bar{b} = 0.7, \gamma = 0.01$. Table \ref{table_runtime} displays the results. It can be seen that when the stopping criterion is set to 0.00001, the other two algorithms cannot converge within the maximum number of iterations while our algorithm can still converge with a small number of iterations. This is because our algorithm has a linear convergence rate. In addition, in all cases, the proposed privacy-preserving algorithm requires less running time than the other two algorithms. Even under the same number of iterations, the running time required by our algorithm is still less than that of the other two algorithms, as shown in Table \ref{table_runtime_undersameiterations}. This is because, compared to the Paillier cryptosystem used in the other two algorithms, the proposed algorithm employs AES, which has a faster encryption and decryption speed.
%This is because the proposed privacy-preserving algorithm reduces the number of encryption and decryption operations per iteration.

\begin{table*}[!t]
\centering
\caption{Running time/number of iterations of our algorithm, PP-DPGD \cite{zhang2019enabling}, and PP-ADMM \cite{zhang2019admm} under different stopping criteria and data dimensions.}
\begin{threeparttable}
\begin{tabular}{|c|cc|cc|cc|}
\hline
\multirow{2}{*}{\begin{tabular}[c]{@{}c@{}}Stopping\\ criterion\end{tabular}} & \multicolumn{2}{c|}{Our algorithm}           & \multicolumn{2}{c|}{PP-DPGD}                    & \multicolumn{2}{c|}{PP-ADMM}                     \\ \cline{2-7} 
                                                                              & \multicolumn{1}{c|}{(3,2)}      & (9,6)      & \multicolumn{1}{c|}{(3,2)}        & (9,6)       & \multicolumn{1}{c|}{(3,2)}       & (9,6)         \\ \hline
0.01                                                                          & \multicolumn{1}{c|}{0.63/42}   & 1.94/58   & \multicolumn{1}{c|}{1923.05/551}   & 2304.54/220 & \multicolumn{1}{c|}{486.82/115}  & 278.86/22     \\ \hline
0.001                                                                         & \multicolumn{1}{c|}{1.06/74}   & 3.37/98  & \multicolumn{1}{c|}{3730.28/1069} & 4104.41/392 & \multicolumn{1}{c|}{1731.26/408}  & 3173.84/250   \\ \hline
0.0005                                                                        & \multicolumn{1}{c|}{1.21/86}   & 4.05/118  & \multicolumn{1}{c|}{4469.49/1281} & 4753.63/454 & \multicolumn{1}{c|}{2293.36/540} & 6714.14/530   \\ \hline
0.0001                                                                        & \multicolumn{1}{c|}{1.63/116}  & 5.61/159 & \multicolumn{1}{c|}{9867.28/2827} & 6792.51/649 & \multicolumn{1}{c|}{**/10000 \tnote{a}}    & 29600.10/2331 \\ \hline
0.00001                                                                       & \multicolumn{1}{c|}{2.10/149} & 7.12/205 & \multicolumn{1}{c|}{**/10000 \tnote{a}}     & **/10000 \tnote{a}    & \multicolumn{1}{c|}{**/10000 \tnote{a}}    & **/10000 \tnote{a}     \\ \hline
\end{tabular}
\begin{tablenotes}
      \item[a] The algorithm reaches the maximum number of iterations, and its running time is very high (more than several hours).
\end{tablenotes}
\end{threeparttable}
\label{table_runtime}  
\end{table*}

% \usepackage{multirow}
\begin{table}[]
\centering
\caption{Running time of our algorithm, PP-DPGD \cite{zhang2019enabling}, and PP-ADMM \cite{zhang2019admm} under the same number of iterations, unit of seconds.}
\begin{tabular}{|c|cc|cc|cc|}
\hline
\multirow{2}{*}{Iterations} & \multicolumn{2}{c|}{Our algorithm}   & \multicolumn{2}{c|}{PP-DPGD}         & \multicolumn{2}{c|}{PP-ADMM}          \\ \cline{2-7} 
                            & \multicolumn{1}{c|}{(3,2)}  & (9,6)  & \multicolumn{1}{c|}{(3,2)}  & (9,6)  & \multicolumn{1}{c|}{(3,2)}  & (9,6)   \\ \hline
100                         & \multicolumn{1}{c|}{1.48}  & 3.46 & \multicolumn{1}{c|}{348.13} & 1045.89 & \multicolumn{1}{c|}{427.13} & 1272.43  \\ \hline
200                         & \multicolumn{1}{c|}{2.92} & 7.10 & \multicolumn{1}{c|}{697.64} & 2106.14 & \multicolumn{1}{c|}{857.30} & 2533.72 \\ \hline
\end{tabular}
\label{table_runtime_undersameiterations}  
\end{table}
\section{Conclusions}\label{section_conclusion}
In this paper, we propose an AES-based privacy-preserving algorithm for the distributed optimization problem. We show that the proposed algorithm can protect the gradients of agents against both the honest-but-curious adversary and the external eavesdropper, only requiring agents to have at least one legitimate neighbor at the initial iteration. Under the assumption that the objective function is strongly convex and Lipschitz smooth, we prove that the proposed algorithm has a linear convergence rate. Our algorithm can be applied to time-varying directed graphs by appropriately constructing the underlying weight matrices. Our algorithm also reduces the computational and communication overhead by using computationally efficient AES and reducing the number of interactions between agents.

The proposed algorithm can only preserve the gradients of agents. For future research, it is worthwhile to develop algorithms that preserve both the gradients and intermediate states with respect to the decision variable.
\appendices
\section{}\label{proof_lemmas}
\subsection{Proof of Lemma \ref{lemma_vr}}
\begin{proof}
It follows from Assumption \ref{smooth} that
	\begin{equation}
	\label{db_1}
	\begin{aligned}
	&\ \| \nabla \mathbf{f}(\mathbf{x}(k+1)) - \nabla \mathbf{f}(\mathbf{x}(k)) \|_{\text{F}} \\
	\le &\ \hat{L} \| \mathbf{x}(k+1) - \mathbf{x}(k) \|_{\text{F}} \\
	 = &\ \hat{L} \| (\mathbf{x}(k+1) - \mathbf{x}^*) - (\mathbf{x}(k) - \mathbf{x}^*) \|_{\text{F}} \\
	 \le &\ \hat{L} \| \mathbf{x}(k+1) - \mathbf{x}^* \|_{\text{F}} + \hat{L} \| \mathbf{x}(k) - \mathbf{x}^* \|_{\text{F}}.
	\end{aligned}
	\end{equation}
Based on the definitions of $\mathbf{v}(k)$ and $\mathbf{r}(k)$ and multiplying both sides of (\ref{db_1}) by $\theta^{-(k+1)}$, we have
	\begin{equation}
	\label{db_2}
	\theta^{-(k+1)} \| \mathbf{v}(k+1) \|_{\text{F}} \le \hat{L} \theta^{-(k+1)} \| \mathbf{r}(k+1) \|_{\text{F}} + \frac{\hat{L}}{\theta}\theta^{-k} \| \mathbf{r}(k) \|_{\text{F}}.
	\end{equation}
	
Finally, taking $\max_{k = 1, \dots, K-1}\{\cdot\}$ on both sides of (\ref{db_2}), we can obtain (\ref{db_0}).
\end{proof}

\subsection{Proof of Lemma \ref{lemma_uv}}
Before giving the proof of Lemma \ref{lemma_uv}, we present the following auxiliary lemma.
\begin{lemma}\label{contraction}
Let $B_{0}$ and $\varepsilon$ be given by (\ref{equation_varepsilon}). For any matrix $D$ with appropriate dimensions, if $C = \Phi_{B_0}(k) D,\ k = B_0, B_0 + 1, \dots$, we have $\| C \|_{\text{R}} \le \varepsilon \| D \|_{\text{R}}$.
\end{lemma}
%\begin{lemma}\label{contraction}
%Under Assumption \ref{connection} and the weight generation rule in Table \ref{table_weight_rule}, let $B_{0}$ be a positive integer such that $B_0 \ge B$ and
%	\begin{equation}
%	\varepsilon = \epsilon (1 - \sigma^{m B})^{(B_0 - 1)/(m B)} < 1,
%	\end{equation}
%where $\epsilon = 2m\frac{1+\sigma^{- m B}}{1-\sigma^{m B}}$ and $\sigma = (c_0)^{2+m B}$. Then, for any matrix $D$ with appropriate dimensions, if $C = \Phi_{B_0}(k) D,\ k = B_0, B_0 + 1, \dots$, we have $\| C \|_{\text{R}} \le \varepsilon \| D \|_{\text{R}}$.
%\end{lemma}
\begin{proof}
We know that $\Phi(k) = W(k+1)^{-1}A(k)W(k)$ and $\Phi(k)$ is a row stochastic matrix with each entry being $\Phi_{ij}(k) = a_{ij}(k)w_j(k)/w_i(k+1)$. We further denote the stochastic vector $\phi(k, B_0)= (\mathbf{1}^T \Phi_{B_0}(k))^T/m$ with $k \ge B_0$, then
	\begin{equation}\label{CD_1}
	\begin{aligned}
	\| C \|_{\text{R}} &= \| (I - \frac{1}{m}\mathbf{1}\mathbf{1}^T) \Phi_{B_0}(k) D \|_{\text{F}} \\
	&= \| (\Phi_{B_0}(k) - \mathbf{1}( \phi(k, B_0) )^T)(I - \frac{1}{m}\mathbf{1}\mathbf{1}^T) D \|_{\text{F}} \\
	&\le \| \Phi_{B_0}(k) - \mathbf{1}( \phi(k, B_0) )^T \|_2 \| D \|_{\text{R}}.
	\end{aligned}
	\end{equation}
	 
Before analyzing the upper bound of $\| \Phi_{B_0}(k) - \mathbf{1}( \phi(k, B_0) )^T \|_2$, we first give the lower bound of $\Phi_{ij}(k)$. We know that $\Phi_{ij}(k)$ is composed of $a_{ij}(k)$, $w_j(k)$, and $1/w_i(k+1)$. According to the weight generation rule in Table \ref{table_weight_rule}, it is clear that $a_{ij}(k) \ge c_0$ if $(i, j) \in \mathcal{E}(k)$ and $k \ge 1$. We next show that $w_i(k) \ge (c_0)^{m B}$ when $k \ge 1$. Since $[A(k)]_{ii} \ge c_0$ for $k \ge 1$, we can get that for any $i$,
	\begin{equation}
	[A(k)\cdots A(1)]_{ii} \ge c_0[A(k-1)\cdots A(1)]_{ii},\ k\ge 2.
	\end{equation}
Recursively, it follows that $[A(k-1)\cdots A(1)]_{ii} \ge (c_0)^{m B}$ for any $i$ and $2 \le k \le m B + 1$. In addition, it is shown in Lemma 4 of \cite{nedic2015distributed} that for $k \ge m B + 1$ every entry of $A(k-1)\cdots A(1)$ is not less than $(c_0)^{m B}$. Thus, we have $[A(k-1)\cdots A(1)]_{ii} \ge (c_0)^{m B}$ for any $i$ and $k \ge 2$. Furthermore, it follows from (\ref{update_compact_w}) that for any $k \ge 2$,
	\begin{equation}
	\begin{aligned}
	[\mathbf{w}(k)]_i &= [A(k-1)\cdots A(1)\mathbf{w}(1)]_i \\
	&= [A(k-1)\cdots A(1)\mathbf{1}]_i \\
	&\ge [A(k-1)\cdots A(1)]_{ii} \\
	&\ge (c_0)^{m B}.
	\end{aligned}
	\end{equation}
Due to $w_i(1) = 1$ and $0 < c_0 < 1/m$, we have $w_i(k) \ge (c_0)^{mB}$ for all agents $i$ and $k \ge 1$. Combining with the fact that $1/w_i(k) \ge 1/m > c_0$ for $k \ge 1$, we can get that $\Phi_{ij}(k) = a_{ij}(k)w_j(k)/w_i(k+1) > (c_0)^{2+m B}$ for $k \ge 1$.

Then, we can obtain that for $k \ge B_0$,
	\begin{equation}
	\begin{aligned}
	\| \Phi_{B_0}(k) - \mathbf{1}(\phi(k, B_0))^T \|_2 &\le m\| \Phi_{B_0}(k) - \mathbf{1}(\phi(k, B_0))^T \|_{\max} \\
	&\le 2m\frac{1+\sigma^{- m B}}{1-\sigma^{m B}}(1 - \sigma^{m B})^{\frac{B_0 - 1}{mB}},
	\end{aligned}
	\end{equation}
where $\sigma = (c_0)^{2+m B}$ and the second inequality is based on Lemma 4 of \cite{nedic2009distributed}. Finally recalling (\ref{CD_1}), we can conclude Lemma \ref{contraction}.
\end{proof}

Now, we give the proof of Lemma \ref{lemma_uv}.

\begin{proof}
The update (\ref{update_2c}) establishes the relationship between $\mathbf{u}(k+1)$ and $\mathbf{v}(k+1)$. By using (\ref{update_2c}) recursively, we have for $k \ge B_0-1$,
	\begin{equation*}
	\begin{aligned}
	\mathbf{u}(k+1) = &\ \Phi_{B_0}(k) \mathbf{u}(k-B_0+1) \\
	 &+ \Phi_{B_0 - 1}(k)W(k-B_0+2)^{-1}\mathbf{v}(k-B_0+2)\\
	 &+ \cdots + \Phi(k) W(k)^{-1}\mathbf{v}(k) + W(k+1)^{-1}\mathbf{v}(k+1).
	\end{aligned}
	\end{equation*}
By using Lemma \ref{contraction}, it follows that for $k \ge B_0$,
	\begin{equation}\label{ud_1}
	\begin{aligned}
	&\ \| \check{\mathbf{u}}(k+1) \|_{\text{F}} \\
	 =&\ \| \mathbf{u}(k+1) \|_{\text{R}}\\
	\le &\ \| \Phi_{B_0}(k) \mathbf{u}(k-B_0+1) \|_{\text{R}} \\
	&+ \| \Phi_{B_0 - 1}(k)W(k-B_0+2)^{-1}\mathbf{v}(k-B_0+2) \|_{\text{R}} \\
	&+ \cdots + \| W(k+1)^{-1}\mathbf{v}(k+1) \|_{\text{R}} \\
	\le &\ \varepsilon \| \check{\mathbf{u}}(k-B_0+1) \|_{\text{F}} \\
	&+ \epsilon \sum_{i = 1}^{B_0} \| W(k-i+2)^{-1}\mathbf{v}(k-i+2) \|_{\text{F}} \\
	\le &\ \varepsilon \| \check{\mathbf{u}}(k-B_0+1) \|_{\text{F}} \\
	&+ \epsilon \| W^{-1} \|_{\max}^{1} \sum_{i = 1}^{B_0} \| \mathbf{v}(k-i+2) \|_{\text{F}},
	\end{aligned}
	\end{equation}
where $\| W^{-1} \|_{\max}^{1} = \sup_{k\ge 1} \| W(k)^{-1} \|_{\max} \le 1/(c_0)^{m B}$.

Then, multiplying both sides of (\ref{ud_1}) by $\theta^{-(k+1)}$ for $k = B_0, B_0+1, \dots, K-1$, we have
	\begin{equation}\label{ud_2}
	\begin{aligned}
	&\ \theta^{-(k+1)} \| \check{\mathbf{u}}(k+1) \|_{\text{F}}\\
	\le&\  \varepsilon \theta^{-B_0} \theta^{-(k-B_0+1)} \| \check{\mathbf{u}}(k-B_0+1) \|_{\text{F}} \\
	&+ \epsilon \| W^{-1} \|_{\max}^{1} \sum_{i = 1}^{B_0} \theta^{-(i-1)} \theta^{-(k-i+2)} \| \mathbf{v}(k-i+2) \|_{\text{F}}.
	\end{aligned}
	\end{equation}
In addition, the following inequality always holds for $k = 0, 1, \dots, B_0 - 1$,
	\begin{equation}\label{ud_3}
	\begin{aligned}
	\theta^{-(k+1)} \| \check{\mathbf{u}}(k+1) \|_{\text{F}} \le \theta^{-(k+1)} \| \check{\mathbf{u}}(k+1) \|_{\text{F}}.
	\end{aligned}
	\end{equation}
Furthermore, taking $\max_{k = B_0, B_0+1, \dots, K-1} \{ \cdot \}$ for (\ref{ud_2}) and $\max_{k = 0, 1, \dots, B_0-1} \{ \cdot \}$ for (\ref{ud_3}), and combining the resulting two inequalities, it obtains
	\begin{equation}
	\begin{aligned}
	\| \check{\mathbf{u}} \|_{\text{F}}^{\theta, K} \le &\ \frac{\varepsilon}{\theta^{B_0}} \| \check{\mathbf{u}} \|_{\text{F}}^{\theta, K} + \epsilon \| W^{-1} \|_{\max}^{1} \sum_{i=1}^{B_0} \theta^{-(i-1)} \| \mathbf{v} \|_{\text{F}}^{\theta, K} \\
	& + \sum_{i=1}^{B_0} \theta^{-i} \| \check{\mathbf{u}}(i) \|_{\text{F}}.
	\end{aligned}
	\end{equation}
	
Finally, algebraically transforming the above inequality, we can obtain (\ref{ud_0}).
\end{proof}

\subsection{Proof of Lemma \ref{lemma_xu}}
\begin{proof}
This proof is similar to the proof of Lemma \ref{lemma_uv}. By invoking the update (\ref{update_2b}) recursively, it follows that for $k \ge B_0-1$,
	\begin{equation*}
	\begin{aligned}
	\mathbf{x}(k+1) =&\ \Phi_{B_0}(k) \mathbf{x}(k-B_0+1) - \eta \Phi_{B_0}(k)\mathbf{u}(k-B_0+1) \\
	&- \eta \Phi_{B_0 - 1}(k) \mathbf{u}(k+2-B_0) - \cdots - \eta \Phi(k) \mathbf{u}(k).
	\end{aligned}
	\end{equation*}
By using Lemma \ref{contraction}, we have for $k \ge B_0$,
	\begin{equation}\label{xu_1}
	\begin{aligned}
	\| \check{\mathbf{x}}(k+1) \|_{\text{F}} =&\ \| \mathbf{x}(k+1) \|_{\text{R}} \\
	\le &\ \| \Phi_{B_0}(k) \mathbf{x}(k-B_0+1) \|_{\text{R}} + \eta \| \Phi(k) \mathbf{u}(k) \|_{\text{R}} \\
	&+ \cdots + \eta \| \Phi_{B_0}(k) \mathbf{u}(k-B_0+1) \|_{\text{R}} \\
	\le &\ \varepsilon \| \check{\mathbf{x}}(k-B_0+1) \|_{\text{F}} + \varepsilon \eta \| \check{\mathbf{u}}(k-B_0+1) \|_{\text{F}} \\
	&+ \epsilon \eta \sum_{i=1}^{B_0 - 1}\| \mathbf{u}(k+1-i) \|_{\text{F}}.
	\end{aligned}
	\end{equation}
	
Then, multiplying both sides of (\ref{xu_1}) by $\theta^{-(k+1)}$ for $k = B_0, B_0+1, \dots, K-1$, we have
	\begin{equation}\label{xu_2}
	\begin{aligned}
	&\ \theta^{-(k+1)} \| \check{\mathbf{x}}(k+1) \|_{\text{F}} \\
	 \le&\ \varepsilon \theta^{-B_0}\theta^{-(k-B_0+1)} \| \check{\mathbf{x}}(k-B_0+1) \|_{\text{F}} \\
	&+ \varepsilon \eta \theta^{-B_0}\theta^{-(k-B_0+1)} \| \check{\mathbf{u}}(k-B_0+1) \|_{\text{F}}\\
	&+ \epsilon \eta \sum_{i=1}^{B_0-1} \theta^{-i}\theta^{-(k+1-i)}\| \check{\mathbf{u}}(k+1-i) \|_{\text{F}}.
	\end{aligned}
	\end{equation}
Moreover, we have the following inequality holds for $k = 0, 1, \dots, B_0-1$,
	\begin{equation}\label{xu_3}
	\begin{aligned}
	\theta^{-(k+1)} \| \check{\mathbf{x}}(k+1) \|_{\text{F}} \le \theta^{-(k+1)} \| \check{\mathbf{x}}(k+1) \|_{\text{F}}.
	\end{aligned}
	\end{equation}
Furthermore, taking $\max_{k = B_0, B_0+1, \dots, K-1}\{ \cdot \}$ for (\ref{xu_2}) and $\max_{k = 0, 1, \dots, B_0-1}\{ \cdot \}$ for (\ref{xu_3}), and combining the resulting inequalities, we have
	\begin{equation}\label{xu_4}
	\begin{aligned}
	\| \check{\mathbf{x}} \|_{\text{F}}^{\theta, K} \le&\ \frac{\varepsilon}{\theta^{B_0}}\| \check{\mathbf{x}} \|_{\text{F}}^{\theta, K} + \big(\frac{\varepsilon \eta}{\theta^{B_0}} +\epsilon\eta \sum_{i = 1}^{B_0-1}\theta^{-i} \big) \| \check{\mathbf{u}} \|_{\text{F}}^{\theta, K} \\
	& + \sum_{i=1}^{B_0}\theta^{-i} \| \check{\mathbf{x}}(i) \|_{\text{F}}.
	\end{aligned}
	\end{equation}
	
Finally, based on the above inequality, we can obtain (\ref{xu_0}).
\end{proof}

\subsection{Proof of Lemma \ref{lemma_rx}}
Before the proof of Lemma \ref{lemma_rx}, we give an auxiliary lemma as follows.
\begin{lemma}\label{IGD}
With the same assumptions of Lemma \ref{lemma_rx}, for any $K = 1, 2, \dots$, we have
	\begin{equation}\label{yx_0}
	\begin{aligned}
	\| \bar{y} - x^* \|_{\text{F}}^{\theta, K} \le&\  \frac{1}{\theta \sqrt{m}}\bigg( \sqrt{ \frac{\hat{L}(1+\beta) + \alpha \beta \hat{\mu}}{\bar{\mu} \beta}}\bigg) \sum_{i = 1}^{m} \| \bar{y} - x_i \|_{\text{F}}^{\theta, K} \\
	&+ 2\| \bar{y}(1) - x^* \|_{2},
	\end{aligned}
	\end{equation}
where $\bar{y}(k) = \frac{1}{m}\sum_{i=1}^{m}y_i(k)$.
\end{lemma}
%\begin{lemma}\label{IGD}
%Under Assumptions \ref{smooth} and \ref{convex}, let
%	\begin{equation}
%	\sqrt{1 - \frac{\alpha\eta\bar{\mu}}{1+\alpha}} \le \theta < 1 \quad \text{and}\quad \eta \le \frac{1}{(1+\beta)\bar{L}}
%	\end{equation}
%where $\alpha > 0$ and $\beta > 0$. Then, for any $K = 1, 2, \dots$,
%	\begin{equation}\label{yx_0}
%	\begin{aligned}
%	\| \bar{y} - x^* \|_{\text{F}}^{\theta, K} \le&\  \frac{1}{\theta \sqrt{m}}\bigg( \sqrt{ \frac{\hat{L}(1+\beta) + \alpha \beta \hat{\mu}}{\bar{\mu} \beta}}\bigg) \sum_{i = 1}^{m} \| \bar{y} - x_i \|_{\text{F}}^{\theta, K} \\
%	&+ 2\| \bar{y}(1) - x^* \|_{2},
%	\end{aligned}
%	\end{equation}
%where $\bar{y}(k) = \frac{1}{m}\sum_{i=1}^{m}y_i(k)$.
%\end{lemma}
\begin{proof}
It follows from (\ref{update_d}) that
	\begin{equation}\label{bx_3}
	\mathbf{s}(k+1) = A(k)\mathbf{s}(k) + \nabla \mathbf{f}(\mathbf{x}(k+1)) - \nabla \mathbf{f}(\mathbf{x}(k)).
	\end{equation}
Since $A(k)$ satisfies $\mathbf{1}^T A(k) = \mathbf{1}^T$ for any $k \ge 0$ and $\mathbf{s}(0) = \nabla \mathbf{f}(\mathbf{x}(0))$, we have
	\begin{equation}\label{bx_4}
	\begin{aligned}
	&\ \frac{1}{m}\mathbf{s}(k+1)^T \mathbf{1} - \frac{1}{m}\nabla \mathbf{f}(\mathbf{x}(k+1))^T \mathbf{1}\\
	=&\ \frac{1}{m}\mathbf{s}(k)^T\mathbf{1} - \frac{1}{m}\nabla \mathbf{f}(\mathbf{x}(k))^T\mathbf{1}\\
	=&\ \cdots\\
	=&\ \frac{1}{m}\mathbf{s}(0)^T\mathbf{1} - \frac{1}{m}\nabla \mathbf{f}(\mathbf{x}(0))^T\mathbf{1}\\
	=&\ \mathbf{0}.
	\end{aligned}
	\end{equation}
Thus, we can obtain that $(1/m)\mathbf{s}(k)^T\mathbf{1} = (1/m)\nabla\mathbf{f}(\mathbf{x}(k))^T\mathbf{1}$.

Then, it follows from (\ref{update_a}) that
	\begin{equation}\label{bx_5}
	\begin{aligned}
	\bar{y}(k+1) & = \frac{1}{m}\mathbf{y}(k+1)^T\mathbf{1}\\
	& = \frac{1}{m}\big( \mathbf{y}(k) - \eta\mathbf{s}(k) \big)^T A(k)^T \mathbf{1} \\
	& = \bar{y}(k) - \frac{\eta}{m} \mathbf{s}(k)^T \mathbf{1} \\
	& = \bar{y}(k) - \frac{\eta}{m} \sum_{i=1}^{m} \nabla f_i(x_i(k)).
	\end{aligned}
	\end{equation}
	
Finally, according to Lemma 3.12 in \cite{nedic2017achieving}, we can get (\ref{yx_0}).
\end{proof}

We now give the proof of Lemma \ref{lemma_rx}.

\begin{proof}
It follows from the definition of $\mathbf{r}(k)$ that for any $k\ge1$,
	\begin{equation}\label{bx_1}
	\begin{aligned}
	\mathbf{r}(k) =&\ \mathbf{x}(k) - \mathbf{x}^* \\
	=&\ \mathbf{x}(k) - \mathbf{1}\bar{x}(k)^T + \mathbf{1}\bar{x}(k)^T - \mathbf{1}\bar{y}(k)^T\\
	&+ \mathbf{1}\bar{y}(k)^T - \mathbf{x}^* \\
	= &\ \check{\mathbf{x}}(k) + \frac{1}{m}\mathbf{1}( \mathbf{1} - \mathbf{w}(k) )^T \mathbf{x}(k) + \mathbf{1} ( \bar{y}(k) - x^* )^T,
	\end{aligned}
	\end{equation}
where the third equality holds due to the fact that (\ref{update_c}). Then, taking $\| \cdot \|_{\text{F}}^{\theta, K}$ for (\ref{bx_1}), we have
	\begin{equation}\label{bx_2}
	\begin{aligned}
	\| \mathbf{r} \|_{\text{F}}^{\theta, K} \le \| \check{\mathbf{x}} \|_{\text{F}}^{\theta, K} + \frac{1}{\sqrt{m}} \| ( \mathbf{1} - \mathbf{w} )^T \mathbf{x} \|_{\text{F}}^{\theta, K} + \sqrt{m} \| \bar{y} - x^* \|_{\text{F}}^{\theta, K}.
	\end{aligned}
	\end{equation}
	
It follows from Lemma \ref{IGD} that
	\begin{equation}\label{bx_6}
	\begin{aligned}
	\| \bar{y} - x^* \|_{\text{F}}^{\theta, K} \le&\  \frac{1}{\theta \sqrt{m}}\bigg( \sqrt{ \frac{\hat{L}(1+\beta) + \alpha \beta \hat{\mu}}{\bar{\mu} \beta}}\bigg) \sum_{i = 1}^{m} \| \bar{y} - x_i \|_{\text{F}}^{\theta, K} \\
	&+ 2\| \bar{y}(1) - x^* \|_{2}.
	\end{aligned}
	\end{equation}
Since,
	\begin{equation}\label{bx_7}
	\begin{aligned}
	\sum_{i = 1}^{m} \| \bar{y} - x_i \|_{\text{F}}^{\theta, K} & = \sum_{i = 1}^{m} \| \bar{y} - \bar{x} + \bar{x} - x_i \|_{\text{F}}^{\theta, K} \\
	& \le m \| \bar{y} - \bar{x} \|_{\text{F}}^{\theta, K} + \sum_{i = 1}^{m} \| \bar{x} - x_i \|_{\text{F}}^{\theta, K} \\
	& \le \| (\mathbf{1} - \mathbf{w})^T \mathbf{x} \|_{\text{F}}^{\theta, K} + \sqrt{m} \| \check{\mathbf{x}} \|_{\text{F}}^{\theta, K}.
	\end{aligned}
	\end{equation}
Based on the fact that $\mathbf{1}^T \mathbf{w}(k) = m$ for $k \ge 1$, it follows that for $k \ge 1$,
	\begin{equation}\label{bx_8}
	\begin{aligned}
	\| (\mathbf{1} - \mathbf{w}(k))^T \mathbf{x}(k) \|_{2} & = \| (\mathbf{1} - \mathbf{w}(k))^T ( I - \frac{1}{m}\mathbf{1}\mathbf{1}^T ) \mathbf{x}(k) \|_{2} \\
	& \le \sqrt{m^2 -m} \| \check{\mathbf{x}}(k) \|_{\text{F}} \\
	& \le m \| \check{\mathbf{x}}(k) \|_{\text{F}}.
	\end{aligned}
	\end{equation}
Combining (\ref{bx_6}), (\ref{bx_7}), and (\ref{bx_8}), we have
	\begin{equation}\label{bx_9}
	\begin{aligned}
	\| \bar{y} - x^* \|_{\text{F}}^{\theta, K} \le &\  \frac{1}{\theta \sqrt{m}} \sqrt{ \frac{\hat{L}(1+\beta) + \alpha \beta \hat{\mu}}{\bar{\mu} \beta}} \| (\mathbf{1} - \mathbf{w})^T \mathbf{x} \|_{\text{F}}^{\theta, K} \\
	& + \frac{1}{\theta} \bigg( \sqrt{ \frac{\hat{L}(1+\beta) + \alpha \beta \hat{\mu}}{\bar{\mu} \beta}} \bigg) \| \check{\mathbf{x}} \|_{\text{F}}^{\theta, K} \\
	& + 2\| \bar{y}(1) - x^* \|_{2} \\
	\le &\ \frac{1+\sqrt{m}}{\theta} \bigg( \sqrt{ \frac{\hat{L}(1+\beta) + \alpha \beta \hat{\mu}}{\bar{\mu} \beta}} \bigg) \| \check{\mathbf{x}} \|_{\text{F}}^{\theta, K} \\
	& + 2\| \bar{y}(1) - x^* \|_{2}.
	\end{aligned}
	\end{equation}

Finally, combining (\ref{bx_2}), (\ref{bx_8}), and (\ref{bx_9}), we can obtain (\ref{bx_0}).
\end{proof}

\subsection{Proof of inequality (\ref{equation_theta_B0})}\label{proof_theta_B0}
\begin{proof}
When $0<\theta<1$, $\frac{1-\theta^{B_0}}{1-\theta} \le B_0$ implies $1 - \theta^{B_{0}} \le (1 - \theta)B_{0}$. Denote the function $f(B_{0}) = (1-\theta)B_{0} - 1 + \theta^{B_{0}}$. Its derived function is $f'(B_{0}) = 1-\theta + \theta^{B_{0}}\ln\theta$. When $0<\theta<1$, $\ln\theta<0$ and $\theta^{B_{0}}$ is monotonically decreasing. So, $f'(B_{0})$ is monotonically increasing. Denote the function $g(\theta) = f'(1) = 1-\theta+\theta\ln\theta$. When $0<\theta<1$, $g'(\theta) = \ln\theta < 0$. The function $g(\theta)$ is monotonically decreasing and satisfies $g(\theta) > g(1) = 0$. Moreover, since $f'(B_{0})$ is monotonically increasing, $f'(B_{0})\ge f'(1) = g(\theta)>0$ when $B_0\ge1$. The function $f(B_{0})$ is monotonically increasing. Therefore, we have $f(B_{0})\ge f(1) = 0$ which in turn gives $1 - \theta^{B_{0}} \le (1 - \theta)B_{0}$.
\end{proof}

\section*{ACKNOWLEDGMENT}
The authors would like to thank the anonymous reviewers and associate editor for their constructive comments and insightful suggestions, especially for bringing us the attention on AES, according to which we propose a more efficient algorithm.
%which have improved the quality of this paper. In particular, they have helped us to use AES in the proposed algorithm.

\section*{References}
\bibliographystyle{IEEEtran}%参考文献
\normalem
\bibliography{IEEEabrv, mybib}

\end{document}